\documentclass[12pt]{amsart}
\usepackage{pifont}
\usepackage{txfonts}
\usepackage{}
\usepackage{amsfonts}
\usepackage{graphicx}
\usepackage{epstopdf}
\usepackage{amsmath}
\usepackage{amsthm}
\usepackage{latexsym,bm}
\usepackage{amssymb}
\usepackage{esint}
\usepackage{enumitem}
\usepackage{indentfirst}
\usepackage{amscd}
\newtheorem{thm}{Theorem}[section]
\newtheorem{cor}[thm]{Corollary}
\newtheorem{lem}[thm]{Lemma}

\newtheorem{question}[thm]{Question}

\newtheorem{Remark}{Remark}
\numberwithin{equation}{section}
\numberwithin{Remark}{section}

\begin{document}

\title{A Loewner-Nirenberg phenomenon for Ricci flow on compact manifolds with boundary}

\author{Gang Li$^\dag$}

\begin{abstract}
In this paper, we show that starting from a geodesic ball $\overline{B_{r_0}}(0)$ in $\mathbb{H}^n$, for $n\geq3$, with prescribed non-decreasing rotationally symmetric mean curvature and the fixed conformal class $[g_{\mathbb{S}^{n-1}}]$ on the boundary, the solution $g(t)$ to the normalized Ricci flow $(\ref{equn_NRicciflow})$ which is continuous up to the boundary, exists for all $t>0$ and converges locally uniformly in $B_{r_0}(0)$ to a complete hyperbolic metric as $t\to\infty$(see Theorem \ref{thm_convern345} for details). Moreover, the sectional curvature of $g(t)$ maintains less than $-1$ for $t>0$.
\end{abstract}

\renewcommand{\subjclassname}{\textup{2000} Mathematics Subject Classification}


\thanks{$^\dag$ Research supported by the National Natural Science Foundation of China No. 11701326 and Shandong Provincial Natural Science Foundation No. ZR2026MS0006.}

\address{Gang Li, School of Mathematics, Shandong University, Jinan, Shandong Province, China}
\email{runxing3@gmail.com}

\maketitle


\section{Introduction}


The Ricci flow
\begin{align}\label{equn_Ricciflow}
\frac{\partial}{\partial t}g=-2\text{Ric}_{g}
\end{align}
 introduced by Hamilton in \cite{Ha} evolves the metric in the direction of an Einstein metric under certain renormalization, and has played an important role in the study of the topology of low-dimensional closed manifolds.

In \cite{LY}, the following normalized Ricci flow was used to study the stability of the hyperbolic space $(\mathbb{H}^n,g_{-1})$,
\begin{align}\label{equn_NRicciflow}
\frac{\partial}{\partial t}g=-2(\text{Ric}_g+(n-1)g),
\end{align}
starting from a complete metric $g_0$ near $g_{-1}$. This stability result was later improved and extended in \cite{SSS,Ba1,Ba2,Bah,QSW}, etc. In particular, the stability of conformally compact Einstein (asymptotically hyperbolic Einstein, abbreviated AHE) metrics was studied in \cite{Bah,QSW} using the normalized Ricci flow, under some weaker curvature conditions on the initial metric. For recent developments in this direction, one is referred to \cite{BGIM} and the references therein. Recall that a family of metrics $g(t)$ satisfies $(\ref{equn_Ricciflow})$ if and only if the family of metrics $h(t)$ defined by
\begin{align}\label{equn_vartransf}
h(t)=e^{-2(n-1)t}g(\frac{1}{2(n-1)}(e^{2(n-1)t}-1))
\end{align}
satisfies $(\ref{equn_NRicciflow})$.

On a compact manifold $\overline{M^n}$ with boundary $\partial M$, the initial-boundary value problems for Ricci flow have also been studied, with the boundary remaining umbilic along the flow, aiming to the convergence of the solution to a metric of constant curvature continuous up to the boundary with a totally geodesic boundary, see \cite{Sh,Co,Co1}. The short-time well-posedness of more general initial-boundary problems has been studied in \cite{Pu,Gi,Ch}. In \cite{Gi}, Gianniotis proved that there exists a unique solution in a short time to the Ricci flow starting from a smooth metric $g_0$, with the following boundary conditions:
\begin{align}
&\label{equn_bdc1}H(g)=\eta,\\
&\label{equn_bdc2}[g^T]=[\gamma],
\end{align}
where $H(g)$ is the mean curvature of $\partial M$ and $[g^T]$ is the conformal class of the part of $g$ tangential to the boundary, and the initial and boundary data $g_0$, $\eta$ and $\gamma$ satisfy natural compatibility conditions at $t=0$.

On the other hand, Loewner and Nirenberg \cite{LN} showed that on a bounded domain $\Omega$ of the Euclidean space $\mathbb{R}^n$ for $n\geq3$, there exists a unique complete conformal metric $g$ with scalar curvature $R_g=-n(n-1)$, and it turns out to be asymptotically hyperbolic. Equivalently, there exists a unique solution to the following blow-up boundary value problem (the Loewner-Nirenberg problem, abbreviated as the L-N problem) of the Yamabe equation:
\begin{align}
&\Delta u =\frac{n(n-2)}{4}u^{\frac{n+2}{n-2}},\,\,\,\text{in}\,\,\Omega,\\
&u(p)\to\infty,\,\,\,\text{as}\,\,p\to\partial \Omega.
\end{align}
Later, Aviles and McOwen \cite{AM,AM1} proved the existence of the solution to the L-N problem on a general compact Riemannian manifold $\overline{M}$ with boundary, and the uniqueness of the solution was proved in \cite{ACF}. In \cite{Li1,Li2}, the author introduced the Cauchy-Dirichlet problem of the Yamabe flow, starting from a metric continuous up to the boundary, and showed that when the boundary data tend to infinity not too slowly as the time $t\to\infty$, then the solution converges locally uniformly to the solution of the L-N problem in the interior of $\overline{M}$.

In the study of the AHE manifolds, the analysis of suitable conformal compactifications of the AHE metric usually provides global information and is more convenient. Therefore, it is reasonable to consider the problem of choosing suitable initial-boundary data $g_0,\,\eta$ and $\gamma$ in $(\ref{equn_bdc1})$ and $(\ref{equn_bdc2})$, so that the solution to the normalized Ricci flow $(\ref{equn_NRicciflow})$ starting from a metric $g_0$ continuous up to $\partial M$ converges locally uniformly to an AHE metric in the interior of $M$ as $t\to\infty$, under the boundary condition $(\ref{equn_bdc1})-(\ref{equn_bdc2})$. This is definitely a hard problem. As a first try, we may ask:
\begin{question}
Let $\overline{M}=\overline{B_{r_0}}(0)$ be a closed geodesic ball of radius $r_0>0$ in the hyperbolic space $(\mathbb{H}^n,g_{-1})$, where
\begin{align}
g_{-1}=dr^2+\sinh^2(r)g_{\mathbb{S}^{n-1}},
\end{align}
with $g_{\mathbb{S}^{n-1}}$ the round metric on $\mathbb{S}^{n-1}$, for $0\leq r\leq r_0$. Consider the initial-boundary problem
\begin{align}
&\label{equn_NRF1}\frac{\partial}{\partial t}g=-2(\text{Ric}_g+(n-1)g),\,\,\text{in}\,\,\overline{M}\times[0,\infty),\\
&\label{equn_Nbdc1}H(g)=\eta,\,\,\,\,\,\,\,\,\text{on}\,\,\partial M\times[0,\infty),\\
&\label{equn_Nbdc2}[g^T]=[g_{\mathbb{S}^{n-1}}],\,\,\text{on}\,\,\partial M\times[0,\infty),\\
&\label{equn_Ninc2}g\big|_{t=0}=g_{-1},\,\,\,\,\text{on}\,\,\overline{M},
\end{align}
where $\eta=\eta(t)$ depends only on time $t$. How to choose the boundary data $\eta(t)$ so that the solution to $(\ref{equn_NRF1})-(\ref{equn_Ninc2})$ exists for all time and converges locally uniformly to the complete hyperbolic metric in the interior of $\overline{M}$ as $t\to\infty$?
\end{question}
We now impose the compatibility condition for $(\ref{equn_NRF1})-(\ref{equn_Ninc2})$ with $k\geq2$:
\begin{align}\label{equn_comptk}
\eta(0)=H_{g_{-1}}\big|_{\partial M}=(n-1)\,\frac{\cosh(r_0)}{\sinh(r_0)},\,\,\,\,\,\eta^{(j)}(0)=0,
\end{align}
for $1\leq j\leq k$.

Here is the main theorem of the paper.
\begin{thm}\label{thm_convern345}
Let $n\geq3$ and $k\geq2$. Assume that $\eta=\eta(t)\in C^k([0,\infty))$ is non-decreasing for $t\geq0$ and $\eta'(t)>0$ on $(0,t_0)$ for some $t_0>0$. Assume the compatibility conditions $(\ref{equn_comptk})$ hold for $k$. Then the solution $g(t)$ to $(\ref{equn_NRF1})-(\ref{equn_Ninc2})$ exists for all $t>0$, and  converges locally uniformly in the interior of $\overline{M}$ to the complete hyperbolic metric $g_{\infty}$ on $B_{r_0}(0)$, as $t\to\infty$.
\end{thm}
Combining the variable transformation $(\ref{equn_vartransf})$ with Theorem 1.2 and Theorem 1.3 in \cite{Gi} by Gianniotis, one obtains that there exists a unique solution to $(\ref{equn_NRF1})-(\ref{equn_Ninc2})$ in $C^{2k+\alpha,k+\frac{\alpha}{2}}$ on $\overline{M}\times[0,\epsilon]$ for any $\alpha\in(0,1)$ and some $\epsilon>0$ under the compatibility conditions $(\ref{equn_comptk})$ for some $k\geq2$. Since $g\big|_{t=0}$ and the boundary data are rotationally symmetric, the solution $g(t)$ has the form
\begin{align}\label{equn_metricpolar1}
g(t)=a(r,t)^2dr^2+b(r,t)^2g_{\mathbb{S}^{n-1}},
\end{align}
for $(r,t)\in [0,r_0]\times[0,T)$, and hence the boundary condition $(\ref{equn_Nbdc1})$ is equivalent to the boundary condition
\begin{align}
\text{II}_{g}=\frac{\eta}{n-1}g^T,
\end{align}
with $\text{II}_g$ the second fundamental form of the boundary, as posed in \cite{Sh,Co1,Pu}.

It turns out that the sectional curvature of $g(t)$ maintains less than $-1$ for $t>0$, see $(\ref{inequn_pin1})$ in Section \ref{section_3}, and the solution $g(t)$ exists for all time $t$, see Theorem \ref{thm_existlongtime}. In Section \ref{section_4}, we use a global argument to conclude the locally uniform bounds on the rotationally symmetric metric $g(t)$ for $t>0$, based on which we then show the locally uniform lower bounds on the sectional curvature of $g(t)$ for $t>0$. Thus, by applying the standard regularity theory of the parabolic equations, we obtain that $g(t)$ converges locally uniformly to a rotationally symmetric locally hyperbolic metric $g_{\infty}$ in the interior of $\overline{M}$. Then in Section \ref{section_5}, we show that the volume of $\overline{M}$ with respect to $g(t)$ goes to infinity as $t\to\infty$ based on an integration argument of the evolution equation satisfied by a curvature quantity and the second variation of the volume of $\overline{M}$, and finally prove that $g_{\infty}$ is complete in the interior $M$ of $\overline{M}$. Notice that when $r_0>0$ is sufficiently large, the boundary data $\eta(t)$ could be sufficiently close to $(n-1)$ for $t\geq0$.

\vskip0.2cm
The Ricci flow on compact surfaces with boundary has been better understood, see \cite{LT,Br,CM}. In \cite{Br}, under the condition that the boundary is totally geodesic along the Ricci flow, Brendle showed the convergence of the solution to the normalized Ricci flow to a metric with constant Gauss curvature continuous up to the boundary with totally geodesic boundary; this convergence result was extended to the case when the Gauss curvature of the initial metric is positive and the geodesic curvature on the boundary $k=\psi(t)\geq0$ depends only on $t$ and is non-increasing in $t$.

In \cite{Li3}, on a compact surface $\overline{M^2}$ with boundary, the author employed the initial-boundary value problem, with prescribed geodesic curvature on the boundary, of the normalized Ricci flow $(\ref{equn_NRicciflow})$, which is equivalent to $(\ref{equn_ut2})-(\ref{equn_uint2})$, and studied the locally uniform convergence of the solution to a complete hyperbolic metric in $M$. In particular, he showed that when the boundary geodesic curvature $k_g=\eta(p,t)\geq 1+\epsilon$ for some $\epsilon>0$, there exists $C=C(\epsilon)>0$ such that, if $u_0\geq C(\epsilon)$ on $\overline{M}$ and satisfies the compatibility condition with $\eta$ on $\partial M$, the solution $g(t)$ exists for all $t\geq0$, and converges locally uniformly to the complete hyperbolic metric in $M$. But when the initial condition is $(\ref{equn_Ninc2})$, there's no explicit boundary data $k_g=\eta$ in \cite{Li3} to support the conclusion of such a convergence result. The following theorem gives a partial answer to this problem.

\begin{thm}\label{thm_convern2}
Let $\overline{M}=\overline{B_{r_0}}(0)$ be a closed geodesic ball of radius $r_0>0$ on the hyperbolic space $(\mathbb{H}^2,g_{-1})$. Assume that $\eta=\eta(t)\in C^k([0,\infty))$ is non-decreasing for $t\geq0$ and $\eta'(t)>0$ on $(0,t_0)$ for some $t_0>0$. Moreover, suppose that
 \begin{align}\label{inequn_etaupperbd}
\eta\leq y(t)^{\frac{1}{3}}-2,
\end{align}
 on $\partial M\times[0,\infty)$, where $y(t)\in C^3([0,\infty))$ is some positive function satisfying
\begin{align}\label{inequn_ODEgrowth}
y'\geq 3y+1
\end{align}
for $t\in[0,\infty)$. Assume that there exists a constant $T>0$ and a small constant $\epsilon>0$ such that
\begin{align}\label{inequn_etainfty}
\eta'(t)\geq \epsilon (1+t)^{-1}(\ln(1+t))^{-1}
 \end{align}
 for $t\geq T$.  Then the solution $g(t)$ to $(\ref{equn_NRF1d2})-(\ref{equn_Ninc2d2})$ converges locally uniformly to a complete hyperbolic metric in $M$ as $t\to\infty$, under the compatibility conditions $(\ref{equn_comptk})$ with the mean curvature $H_{g_{-1}}\big|_{\partial M}$ replaced by the geodesic curvature $k_{g_{-1}}\big|_{\partial M}$ for some $k\geq2$.
\end{thm}
Here in the case of dimension two, in order to show the limit metric of the Ricci flow is complete in the interior of $\overline{M}$, we have used the very restrictive inequality $(\ref{inequn_etainfty})$, which is different from the higher dimensional cases and might not be necessary. This can be attributed to the absence of a boundary integral term in the second variation of the volume of $M$ for dimension $2$, in comparison to the case $n\geq3$. Moreover, there are no a priori estimates of the Gaussian curvature of the solution $g(t)$ near $\partial M$ for $n=2$. Hence, we have to combine the conclusion that the Gaussian curvature of $g(t)$ maintains less than $-1$ with the comparison argument in \cite{Li3} to achieve the long-time existence and locally uniform convergence of $g(t)$. For details, see Section \ref{section_6}.

\begin{Remark}
In the appendix, we use ChatGPT 6 Astra to give a proof of Theorem \ref{thm_convern2} without the condition $(\ref{inequn_etainfty})$, and hence Corollary \ref{cor_1.4} holds without the condition $(\ref{inequn_etainfty})$ on $\eta(t)$ for $t$ large.
\end{Remark}

As a direct consequence of Theorem \ref{thm_convern2} here and Theorem 1.3 in \cite{Li3}, we obtain:
\begin{cor}\label{cor_1.4}
Let $\overline{M}=\overline{B_{r_0}}(0)$ be a closed geodesic ball of radius $r_0>0$ on the hyperbolic space $(\mathbb{H}^2,g_{-1})$. Let $\eta=\eta(t)$ satisfy the conditions in Theorem \ref{thm_convern2}. Assume $u_0\in C^{2+\alpha}(\overline{M})$ and $\psi\in C^{1+\alpha,\frac{1}{2}+\frac{\alpha}{2}}(\partial M\times[0,T])$ for all $T>0$, and also the compatibility condition holds on $\partial M\times\{0\}$:
\begin{align}
k_{e^{2u_0}g_{-1}}=\psi(\cdot,0).
\end{align}
Suppose $u_0\geq0$ on $\overline{M}$, and $\psi$ satisfies
\begin{align*}
\eta\leq \psi\leq y_1(t)^{\frac{1}{3}}-2
\end{align*}
on $\partial M\times[0,\infty)$, for a positive function $y_1(t)\in C^3([0,\infty))$ satisfying $(\ref{inequn_ODEgrowth})$ for $t\in[0,\infty)$. Then there exists a unique solution $g(t)$ for all $t>0$ to $(\ref{equn_NRF1})$ with the initial-boundary conditions:
\begin{align*}
&k_{g(t)}=\psi,\,\,\text{on}\,\,\partial M\times[0,\infty),\\
&g(0)=e^{2u_0}g_{-1}\,\,\text{on}\,\,\overline{M}.
\end{align*}
Moreover, $g(t)$ converges locally uniformly to a complete hyperbolic metric in the interior of $\overline{M}$, as $t\to\infty$.
\end{cor}

Throughout this article, the notation $\overline{M}=\overline{B_{r_0}}(0)$ always means the closed $r_0$-ball centered at the origin in the Euclidean space $\mathbb{R}^n$, and the coordinates on $\overline{M}$ will always be expressed using the polar coordinates $(r,\mathbb{S}^{n-1})$ on $\mathbb{R}^n$ with $0\leq r\leq r_0$, and moreover, the starting hyperbolic metric of the Ricci flow on $\overline{M}$ is expressed as
\begin{align}
g\big|_{t=0}=dr^2+\sinh^2(r)g_{\mathbb{S}^{n-1}},
\end{align}
for $0\leq r\leq r_0$, and $(\overline{M}, g(0))$ is the geodesic $r_0$-ball in $\mathbb{H}^n$. The notation $\overline{B_r}(0)$ always means the $r$-ball centered at the origin in $\mathbb{R}^n$, with $B_r(0)$ its interior and $\partial B_r(0)$ its boundary, for $0\leq r\leq r_0$.

\vskip0.2cm
{\bf Statement on the use of AI.} This article does not use AI; however, we have included the proof of Theorem \ref{thm_convern2}, with the condition $(\ref{inequn_etainfty})$ removed, in the appendix using ChatGPT Astra.
\vskip0.2cm
{\bf Acknowledgements.} The author would like to thank Yuxing Deng, Yunrui Zheng and Tao Tao for helpful discussion.

\vskip0.2cm

\section{Preliminaries}

Let $\overline{M}=\overline{B_{r_0}}(0)\subseteq \mathbb{R}^n$, for $n\geq 2$. By the discussion below Theorem \ref{thm_convern345} in the introduction, there exists a unique solution $g(t)$ to $(\ref{equn_NRF1})-(\ref{equn_Ninc2})$ on $\overline{M}\times[0,T)$ with $T>0$ the largest existence time. Moreover, $g(t)$ is rotationally symmetric and of the form:
\begin{align}\label{equn_metricpolar}
g=a(r,t)^2dr^2+b(r,t)^2g_{\mathbb{S}^{n-1}},
\end{align}
in polar coordinates on $\overline{M}$, where for $0\leq r\leq r_0$,
\begin{align}
a(r,0)=1,\,\,\,\,b(r,0)=\sinh(r).
\end{align}
Direct computation shows that the Ricci flow is not a conformal flow for $n\geq3$, unless $g(t)$ has constant sectional curvature for $T>t\geq0$, which is contrary to our conditions. Denote $s$ the distance function to the origin under the metric $g(t)$ for $T>t\geq0$, i.e.,
\begin{align}\label{equn_saformul}
s=s(r,t)=\int_0^ra(x,t)dx,
\end{align}
and hence by the regularity of $g$ at the origin,
\begin{align}\label{equn_originr}
\frac{d}{ds}b(0,t)=1,\,\,\,\frac{d^{2k}}{d s^{2k}}b(0,t)=0,
\end{align}
for $t\geq0$ and any integer $k\geq0$. Therefore, for any fixed $T>t\geq0$, we have
\begin{align*}
&\frac{\partial}{\partial r}=a\,\frac{\partial}{\partial s},\\
&g(t)=ds^2+ b^2g_{\mathbb{S}^{n-1}}.
\end{align*}
By direct calculation, we obtain the formula for the Ricci tensor:
\begin{align}
\text{Ric}_g=-(n-1)b^{-1}b_{ss}''a^2dr^2\,-\,b[b_{ss}''-(n-2)b^{-1}(1-(b_s')^2)]\,g_{\mathbb{S}^{n-1}}.
\end{align}
Since $g(t)$ is rotationally symmetric, by adopting the same notations as in \cite{AK,DG}, we denote the sectional curvatures of the $2$-plane perpendicular to the fibers $\{r\}\times \mathbb{S}^{n-1}$ and of the $2$-planes tangential to the fibers by $K$ and $L$ respectively, and hence the curvature of $g$ is entirely described by $K$ and $L$ with the formulas
 \begin{align}
 K=-b^{-1}b_{ss}'',\,\,\,L=b^{-2}(1-(b_s')^2).
 \end{align}
By $(\ref{equn_originr})$, when $g(t)$ is of $C^{k+2}$ in $\overline{M}$, we have $K(\cdot,t),L(\cdot,t)\in C^k(\overline{M})$. For convenience of later discussion, we introduce the following curvature functions:
\begin{align}
&F=F(r,t)\triangleq-(n-1)b^{-1}b_{ss}''=(n-1)K,\\
&G=G(r,t)\triangleq -b^{-1}b_{ss}''+(n-2)b^{-2}(1-(b_s')^2)=K+(n-2)L,\\
&P=F+n-1,\,\,\,\,Q=G+n-1.
\end{align}
Therefore, the scalar curvature satisfies the formula
\begin{align*}
R_g=F+(n-1)G.
\end{align*}
From $(\ref{equn_NRF1})$, we derive the following evolution equations for $a$ and $b$
\begin{align}
&\label{equn_NRFa}a_t=(n-1)a(b^{-1}b_{ss}''-1) = -P\,a,\\
&\label{equn_NRFb}b_t=b[b^{-1}b_{ss}''-(n-2)b^{-2}(1-(b_s')^2)-n+1]=-Q\,b,
\end{align}
and hence for $(r,t)\in[0,r_0]\times[0,T)$, we have
\begin{align}
&\label{equn_aintt0}a(r,t)=a(r,0)e^{-\int_0^tP(r,\tau)d\tau},\\
&\label{equn_bintt0}b(r,t)=b(r,0)e^{-\int_0^tQ(r,\tau)d\tau}.
\end{align}

\begin{lem}\label{lem_Hp}
If $\eta(t)>0$ for $t\in[0,T)$, then $b_r'(r,t)>0$ for $(r,t)\in[0,r_0]\times[0,T)$.
\end{lem}
\begin{proof}
Recall that the mean curvature of $\partial B_r(0)$ for $0<r\leq r_0$ induced by the metric $g(t)$ is given by the formula
\begin{align}\label{equn_meanH}
H_g(r,t)=(n-1)b^{-1}b_s'.
\end{align}
By $(\ref{equn_NRFb})$ and $(\ref{equn_NRFa})$, we have
\begin{align}\label{equn_bst}
(b_s')_t'&=(a^{-1}b_r')_t'\\
&=(b_s')_{ss}''+(n-3)b^{-1}b_s'(b_s')_s'+(n-2)b^{-2}[1-(b_s')^2]b_s'.\nonumber
\end{align}
By continuity of $b_s'$ for $(r,t)\in[0,r_0]\times[0,T)$, using the boundary condition and $(\ref{equn_originr})$, the maximum principle gives
\begin{align}
b_s'>0,\,\,\,\,\,b_r'=a\,b_s'>0,
\end{align}
for $(r,t)\in[0,r_0]\times[0,T)$.

\end{proof}
For any $T>t\geq0$, using the Taylor expansion of $b$ near $r=0=s$,
\begin{align*}
b(r,t)=s+\frac{s^3}{3!}\partial_s^3b(0,t)+\frac{s^5}{5!}\partial_s^5b(0,t)+O(s^7),
\end{align*}
we get
\begin{align}
Q-P=(n-2)b^{-2}[1-(b_s')^2]+(n-2)b^{-1}b_{ss}''=O(s^2).
\end{align}
In particular,
\begin{align}\label{equn_originPQ}
P(0,t)=Q(0,t)
\end{align}
for $T>t\geq0$. Direct computation yields
\begin{align}\label{equn_Pn-1Qb0}
P-(n-1)Q=-(n-1)(n-2)[b^{-2}(1-(b_s')^2)+1]=-(n-1)(n-2)[b^{-2}+1-\frac{H_g^2}{(n-1)^2}]
\end{align}
where $H_g=H_g(r,t)$ is the mean curvature of $\partial B_r(0)$ with respect to the metric $g(t)$, and hence,
\begin{align}\label{equn_Pn-1Qs}
\frac{\partial}{\partial s}(P-(n-1)Q)=-2(P-Q)H_g,
\end{align}
which will play an important role in the discussion on the signs of $P$ and $Q$. In particular,
\begin{align}
&\label{equn_Pn-1Qb}P-(n-1)Q=-(n-1)(n-2)[b^{-2}+1-\frac{\eta(t)^2}{(n-1)^2}],\\
&\label{equn_nPn-1Qb}\frac{\partial}{\partial n_g}(P-(n-1)Q)=-2(P-Q)\,\eta(t),
\end{align}
on $\partial M\times[0,T)$, where $n_g$ is the outer unit normal vector field. Differentiating $H_g(r,t)$ with respect to $t$, and substituting $(\ref{equn_NRFb})$ and $(\ref{equn_bst})$ into the result, we obtain
\begin{align}
\partial_tH_g=H_gP-(n-1)\partial_sQ,
\end{align}
and hence, we have
\begin{align}\label{equn_DtHb}
\eta'(t)=\eta \,P\,-\,(n-1)\frac{\partial}{\partial n_g}Q
\end{align}
on $\partial M\times[0,T)$.

We now compute how $P$ and $Q$ evolve under the flow. Similar to $(\ref{equn_bst})$, we obtain
\begin{align*}
(b_{ss}'')_t'=&\,(b_{ss}'')_{ss}''+(n-3)b^{-1}b_s'(b_{ss}'')_s'-2b^{-1}(b_{ss}'')^2-(4n-9)b^{-2}(b_s')^2b_{ss}''\\
&+(n-2)b^{-2}b_{ss}''+(n-1)b_{ss}''-2(n-2)b^{-3}(b_s')^2(1-(b_s')^2).
\end{align*}
Recall that for a rotationally symmetric function $f$ the Laplacian acting on $f$ satisfies that
\begin{align*}
\Delta_{g(t)}f=f_{ss}''+(n-1)b^{-1}b_s'f_s'.
\end{align*}
Therefore, for $(r,t)\in(0,r_0]\times[0,T)$ and dimension $n\geq3$,
\begin{align}
&\label{equn_envP}\partial_tP=\Delta_gP+2P(P+1-n)-2(n-1)b^{-2}(P-Q)-\frac{2(n-1)}{n-2}(P-Q)^2,\\
&\label{equn_envQ}\partial_tQ=\Delta_gQ+2Q(Q+1-n)+2b^{-2}(P-Q)+\frac{2}{n-2}(P-Q)^2,
\end{align}
where $g=g(t)$.

For $n=2$, we have
\begin{align}\label{equn_PQn2}
P=Q=-b^{-1}b_{ss}''+1=K_g+1=\frac{1}{2}(R_g+2),
\end{align}
where $K_g$ is the Gaussian curvature and $R_g$ is the scalar curvature and hence,
\begin{align}\label{equn_n2Pt}
\partial_tP=\Delta_gP+2P(P-1).
\end{align}
The geodesic curvature $k_g$ of $\partial B_r(0)$ with respect to $g(t)$ has the expression $k_g=b(r,t)^{-1}\partial_sb(r,t)$. Thus we obtain
\begin{align}
\partial_tk_g=k_gP-\partial_sP,
\end{align}
and hence,
\begin{align}\label{equn_Dtkb}
\eta'(t)=\eta \,P\,-\,\frac{\partial}{\partial n_g}P
\end{align}
on $\partial M\times[0,T)$. The initial-boundary problem $(\ref{equn_NRF1})-(\ref{equn_Ninc2})$ reduces to the following form
\begin{align}
&\label{equn_NRF1d2}\frac{\partial}{\partial t}g=-2(K_g+1)g,\,\,\text{in}\,\,\overline{M}\times[0,\infty),\\
&\label{equn_Nbdc1d2}k_g=\eta,\,\,\,\,\,\,\,\,\text{on}\,\,\partial M\times[0,\infty),\\
&\label{equn_Ninc2d2}g\big|_{t=0}=g_{-1},\,\,\,\,\text{on}\,\,\overline{M}.
\end{align}
By $(\ref{equn_aintt0})$, $(\ref{equn_bintt0})$ and $(\ref{equn_PQn2})$, we obtain
\begin{align}
&\label{equn_aPn2}a(r,t)=a(r,0)e^{-\int_0^tP(r,\tau)d\tau},\\
&\label{equn_bPn2}b(r,t)=b(r,0)e^{-\int_0^tP(r,\tau)d\tau},
\end{align}
for $(r,t)\in[0,r_0]\times[0,T)$. Let $u(r,t)=-\int_0^tP(r,\tau)d\tau$, for $(r,t)\in [0,r_0]\times[0,T)$. Therefore,
\begin{align*}
g(t)=e^{2u}g(0),
\end{align*}
for $t\in[0,T)$, and hence, the initial-boundary value problem $(\ref{equn_NRF1d2})-(\ref{equn_Ninc2d2})$ is reduced to the following (see also \cite{Li3})
\begin{align}
&\label{equn_ut}u_t=e^{-2u}(\Delta_{g(0)}u-K_{g(0)})-1,\,\,\text{in}\,\,\overline{M}\times[0,\infty),\\
&\label{equn_unb}\frac{\partial u}{\partial n_{g(0)}}+k_{g(0)}=\eta e^u,\,\,\text{on}\,\,\partial M\times[0,\infty),\\
&\label{equn_uint}u\big|_{t=0}=0,\,\,\,\text{in}\,\,\overline{M}.
\end{align}

\vskip0.2cm

\section{Preserving curvature upper bounds along the normalized Ricci flow}\label{section_3}
Let $\overline{M}$ be of dimension $n\geq3$. In this section we will show that the curvature functions defined in the preliminaries satisfy
\begin{align}\label{inequn_pin1}
(n-1)Q\leq P\leq Q\leq 0,
\end{align}
in $\overline{M}\times[0,+\infty)$ and the solution $g(t)$ exists for all time $t\geq0$. First, we show that $(\ref{inequn_pin1})$ holds for $t\geq0$ small.

\begin{lem}\label{lem_P_Q-0}
Let $n\geq3$. Let $g(t)$ be a solution to $(\ref{equn_NRF1})-(\ref{equn_Ninc2})$ on $\overline{M}\times[0,T)$ for some $T>0$ under the compatibility conditions $(\ref{equn_comptk})$ for some $k\geq2$. Assume that $\eta=\eta(t)\in C^k([0,\infty))$ is non-decreasing for $t\geq0$ and $\eta'(t)>0$ on $(0,t_0)$ for some $t_0>0$.  Then there exists $\epsilon\in(0,T)$ such that $(\ref{inequn_pin1})$ holds on $M\times[0,\epsilon]$.
\end{lem}
\begin{proof}
Subtracting $(\ref{equn_envQ})$ from $(\ref{equn_envP})$ yields
\begin{align}\label{equn_P-Qevol}
\partial_t(P-Q)=\Delta_g(P-Q)+2(P-Q)[1-n-nb^{-2}-\frac{2}{n-2}(P-(n-1)Q)],
\end{align}
on $\overline{M}\times[0,T)$. Recall that $P(r,0)=Q(r,0)=0$ for $0\leq r\leq r_0$. By continuity of $P,\,Q$ and $b$, we have that there exists $\epsilon_1>0$ small such that
\begin{align*}
1-n-nb^{-2}-\frac{2}{n-2}(P-(n-1)Q)<0
\end{align*}
for $0<r\leq r_0$ and $0\leq t\leq \epsilon_1$. On the other hand, by $(\ref{equn_originPQ})$, $P(0,t)-Q(0,t)=0$ for $t\geq0$, and hence, by the maximum principle for $(\ref{equn_P-Qevol})$ and continuity of $P-Q$, we have that the positive maximum and negative minimum of $P-Q$ on $\overline{M}\times[0,T_1]$ for any $T_1\in(0,T)$ can only be attained on the boundary $\partial M\times[0,T_1]$. Therefore, once it is proved that $P-Q\leq0$ on $\partial M\times[0,T_1]$ for some $T_1\in(0,T)$, it follows that $P-Q\leq0$ on $\overline{M}\times[0,T_1]$; if, in addition, $P-Q$ is not identically zero on $\partial M\times[0,T_1]$, by the strong maximum principle, we have that $P(r,t)-Q(r,t)<0$ for $(r,t)\in(0,r_0)\times(0,T_1]$.

Denote $b_0=b(r_0,0)=\sinh(r_0)$. For $t>0$ small, by $(\ref{equn_bintt0})$ and the Taylor expansion, we have
\begin{align}
b(r,t)=b(r,0)e^{-\int_0^tQ(r,\tau)d\tau}=b(r,0)[1\,-\,(1+o(1))\,\int_0^tQ(r,\tau)d\tau]
\end{align}
where $o(1)\to0$ uniformly for $r\in[0,r_0]$ as $t\to0$. Thus, by $(\ref{equn_Pn-1Qb})$, $(\ref{equn_bintt0})$ and $(\ref{equn_comptk})$,
\begin{align}\label{equn_Pn-1Qb1}
&P(r_0,t)-(n-1)Q(r_0,t)\\
=&-(n-1)(n-2)[1+b_0^{-2}e^{2\int_0^tQ(r_0,\tau)d\tau}]+\frac{n-2}{n-1}\eta(t)^2\nonumber\\
=&-(n-1)(n-2)[1+b_0^{-2}(1+2(1+o(1))\int_0^tQ(r_0,\tau)d\tau)]+\frac{n-2}{n-1}(\eta(0)+\int_0^t\eta'(\tau)d\tau)^2\nonumber\\
=&-2(n-1)(n-2)b_0^{-2}\int_0^tQ(r_0,\tau)d\tau\,(1+o(1))\,+\,\frac{2(n-2)}{n-1}\eta(0)\int_0^t\eta'(\tau)d\tau\,(1+o(1)),\nonumber
\end{align}
where $o(1)\to0$ as $t\to0$. Therefore, by $(\ref{equn_nPn-1Qb})$ and $(\ref{equn_DtHb})$ we have
\begin{align}\label{equn_nP-Qb}
&\frac{\partial}{\partial n_g}(P-Q)\\
=&(-P+Q)\eta-\frac{\eta}{n-1}(P-(n-1)Q)-\frac{n-2}{n-1}\eta'(t)\nonumber\\
=&(Q-P)\eta+2(n-2)b_0^{-2}\eta(0)\int_0^tQ(r_0,\tau)d\tau\,(1+o(1))-\,\frac{2(n-2)}{(n-1)^2}\eta(0)^2\int_0^t\eta'(\tau)d\tau\,(1+o(1))\nonumber\\
&-\frac{n-2}{n-1}\eta'(t)\nonumber\\
\equiv&I_1+I_2+I_3+I_4,\nonumber
\end{align}
on $\partial M$ for $t>0$ small, where $o(1)\to0$ as $t\to0$.

{\bf Claim:} there exists $\epsilon_2\in(0,T)$ such that $P-Q\leq0$ on $\partial M\times[0,\epsilon_2]$.

Assume the contrary, i.e., for each $\delta>0$ small, there exists $t\in(0,\delta)$ such that
\begin{align}
P(r_0,t)-Q(r_0,t)>0.
\end{align}
Thus, for each $\delta>0$, 
there exists $t_1\in(0,\delta)$, such that
\begin{align}
&\label{inequn_P-Qbm}(P-Q)(r_0,t_1)=\sup_{0\leq t\leq t_1}(P-Q)(r_0,t)>0,\\
&\label{inequn_nP-Q0}\frac{\partial}{\partial n_g}(P-Q)(r_0,t_1)\geq0.
\end{align}
and hence, the first equality in $(\ref{equn_nP-Qb})$ gives
 \begin{align}\label{inequn_Qb01}
 Q(r_0,t_1)\geq \frac{\eta'(t_1)}{\eta(t_1)}>0.
 \end{align}
 Also, it is clear that $I_i<0$ in $(\ref{equn_nP-Qb})$ at $(r_0,t_1)$ for $i=1,3,4$. Therefore, by $(\ref{inequn_nP-Q0})$, $I_2+I_3+I_4>0$ at $(r_0,t_1)$, i.e.,
\begin{align}\label{inequn_P-QQb}
\int_0^{t_1}Q(r_0,\tau)d\tau>\frac{b_0^2\eta'(t_1)}{2(n-1)\eta(0)}\,(1+o(1))+\frac{b_0^2\eta(0)}{(n-1)^2}\int_0^{t_1}\eta'(\tau)d\tau\,(1+o(1))>0.
\end{align}
Thus, by continuity of $Q$, there exists a smallest $t_2\in(0,t_1]$, such that
\begin{align}\label{inequn_Qb02}
Q(r_0,t_2)=\sup_{0 \leq t\leq t_1}Q(r_0,t)> \frac{b_0^2\eta'(t_1)}{2(n-1)t_1\eta(0)}\,(1+o(1))>0.
\end{align}

 Combining $(\ref{inequn_P-Qbm})$,  $(\ref{equn_Pn-1Qb1})$ and $(\ref{inequn_Qb02})$, we have
\begin{align}\label{inequn_P-Qb01}
(P-Q)(r_0,t_1)\geq&(P-Q)(r_0,t_2)\\
=&(n-2)Q(r_0,t_2)-2(n-1)(n-2)b_0^{-2}\int_0^{t_2}Q(r_0,\tau)d\tau\,(1+o(1))\nonumber\\
&+\frac{2(n-2)\eta(0)}{n-1}\int_0^{t_2}\eta'(\tau)d\tau\,(1+o(1))\nonumber\\
=&(n-2)Q(r_0,t_2)\,(1+o(1))+\frac{2(n-2)\eta(0)}{n-1}\int_0^{t_2}\eta'(\tau)d\tau\,(1+o(1))\nonumber\\
\geq&(n-2)Q(r_0,t_1)\,(1+o(1))+\frac{2(n-2)\eta(0)}{n-1}\int_0^{t_2}\eta'(\tau)d\tau\,(1+o(1))\nonumber\\
\geq&\frac{(n-2)\eta'(t_1)}{\eta(0)}\,(1+o(1))+\frac{2(n-2)\eta(0)}{n-1}\int_0^{t_2}\eta'(\tau)d\tau\,(1+o(1)),\nonumber
\end{align}
with $o(1)\to0$ as $\delta\to0$, where for the last inequality we have used $(\ref{inequn_Qb01})$.

 Now at $t=t_1$, by $(\ref{inequn_P-Qb01})$, the terms $I_1$ and $I_2$ in $(\ref{equn_nP-Qb})$ satisfy
\begin{align*}
I_1&=(Q-P)(r_0,t_1)\,\eta(0)\,(1+o(1))\\
&\leq-(n-2)Q(r_0,t_2)\eta(0)\,(1+o(1))-\frac{2(n-2)\eta(0)^2}{n-1}\int_0^{t_2}\eta'(\tau)d\tau\,(1+o(1))\\
&\leq-(n-2)Q(r_0,t_2)\eta(0)\,(1+o(1)),\\
I_2&=2(n-2)b_0^{-2}\eta(0)\int_0^{t_1}Q(r_0,\tau)d\tau\,(1+o(1))\\
&\leq 2(n-2)t_1b_0^{-2}\eta(0)Q(r_0,t_2)(1+o(1)),
\end{align*}
with $o(1)\to0$ and $t_1\to0$ as $\delta\to0$, and hence, by $(\ref{inequn_Qb02})$, $(\ref{equn_nP-Qb})$ and the fact $I_i<0$ for $i=3,4$, we have
\begin{align}
\frac{\partial}{\partial n_g}(P-Q)(r_0,t_1)<0,
\end{align}
for $\delta>0$ sufficiently small, contradicting with the assumption that $(P-Q)$ achieves its maximum on $\overline{M}\times[0,t_1]$ at the point $(r_0,t_1)$. This proves the {\bf Claim}. Therefore, we obtain
\begin{align*}
P(r,t)-Q(r,t)\leq 0
\end{align*}
for $\overline{M}\times[0,\epsilon_2]$, and moreover,
\begin{align}\label{inequn_P-Q-0}
P(r,t)-Q(r,t)< 0
\end{align}
for $(r,t)\in(0,r_0)\times(0,\epsilon_2]$.

We now show that $Q(r,t)\leq 0$ on $\overline{M}\times [0,\epsilon_3]$ for some $\epsilon_3\in(0,\epsilon_2]$. First, we verify that it holds on the boundary. By $(\ref{equn_Pn-1Qb1})$, we have
\begin{align}\label{equn_QbODE0}
&Q(r_0,t)-2(n-1)(1+o(1))\,b_0^{-2}\int_0^tQ(r_0,\tau)d\tau\,\\
=&\frac{1}{(n-2)}(P-Q)(r_0,t)-\,\frac{2}{n-1}(1+o(1))\eta(0)\int_0^t\eta'(\tau)d\tau<0,\nonumber
\end{align}
for $t\in(0,\epsilon_2]$, where $o(1)\to0$ as $t\to0$, and the right-hand side of the first equality is now denoted by $f(t)$. We denote $\xi(t)=\int_0^tQ(r_0,\tau)d\tau$ and get the following first order differential equation for $\xi(t)$ for $t\in(0,\epsilon_3]$:
\begin{align}\label{equn_xibODE}
\xi'(t)-2(n-1)(1+o(1))\,b_0^{-2} \xi(t)=f(t)<0,
\end{align}
and hence by solving $(\ref{equn_xibODE})$, we obtain
\begin{align}
\xi(t)=e^{2(n-1)b_0^{-2}\int_0^t(1+o(1))ds}[\int_0^te^{-2(n-1)b_0^{-2}\int_0^{\tau}(1+o(1))ds}f(\tau)d\tau]<0
\end{align}
for $t\in(0,\epsilon_3]$, and hence by $(\ref{equn_xibODE})$,
\begin{align}\label{inequn_Qb03}
Q(r_0,t)=\xi'(t)=2(n-1)(1+o(1))\,b_0^{-2} \xi(t)+f(t)<0,
\end{align}
for $t\in(0,\epsilon_3]$.

On the other hand, recall that $P(r,0)=Q(r,0)=0$ for $0\leq r\leq r_0$. Since $(P-Q)\leq0$ on $\overline{M}\times[0,\epsilon_2]$, by continuity of $P$ and $Q$ and the equation $(\ref{equn_envQ})$, we have that there exists $\epsilon_4\in(0,\epsilon_3]$ such that
\begin{align}\label{inequn_Qt}
\partial_tQ\leq \Delta_gQ+2(Q+1-n)Q
\end{align}
on $\overline{M}\times[0,\epsilon_4]$. By the maximum principle for $(\ref{inequn_Qt})$ and $(\ref{inequn_Qb03})$, we have
\begin{align}\label{inequn_Q-0}
Q<0,\,\,\,\text{on}\,\,\overline{M}\times(0,\epsilon_4].
\end{align}
Thus, by $(\ref{equn_Pn-1Qb1})$, we have
\begin{align*}
(P-(n-1)Q)(r_0,t)>0
\end{align*}
for $t\in(0,\epsilon_5)$, with some $\epsilon_5\in(0,\epsilon_4)$. By $(\ref{inequn_Q-0})$ and $(\ref{equn_originPQ})$, we have
\begin{align}
(P-(n-1)Q)(0,t)=-(n-2)Q(0,t)>0
\end{align}
for $t\in(0,\epsilon_4]$; on the other hand, by $(\ref{inequn_P-Q-0})$, $(\ref{equn_Pn-1Qs})$ and Lemma \ref{lem_Hp}, we get
\begin{align}
\frac{\partial}{\partial s}(P-(n-1)Q)=-2H_g(P-Q)>0
\end{align}
for $(r,t)\in(0,r_0)\times(0,\epsilon_2)$. Therefore, we have
\begin{align}
P>(n-1)Q
\end{align}
for $(r,t)\in[0,r_0]\times(0,\epsilon_4]$. Now take $\epsilon=\epsilon_4>0$.

In summary, we finally obtain that
\begin{align}
(n-1)Q<P\leq Q<0\,\,\,\text{on}\,\,\overline{M}\times(0,\epsilon],
\end{align}
and moreover,
\begin{align*}
P< Q
\end{align*}
for $(r,t)\in (0,r_0)\times (0,\epsilon]$, and hence by the first equality in $(\ref{equn_nP-Qb})$, we have
\begin{align*}
P< Q
\end{align*}
for $(r,t)\in (0,r_0]\times (0,\epsilon]$. This completes the proof of the lemma.

\end{proof}

Now we are ready to show that $(\ref{inequn_pin1})$ holds for all time.

\begin{thm}\label{thm_PQn-1Q2}
Let $n\geq3$. Let $g(t)$ be a solution to $(\ref{equn_NRF1})-(\ref{equn_Ninc2})$ on $\overline{M}\times[0,T)$ for some $T>0$ under the compatibility conditions $(\ref{equn_comptk})$ for some $k\geq2$. Assume that $\eta=\eta(t)\in C^k([0,\infty))$ is non-decreasing for $t\geq0$ and $\eta'(t)>0$ on $(0,t_0)$ for some $t_0>0$.  Then $(\ref{inequn_pin1})$ holds on $M\times[0,T)$. Indeed, we have
\begin{align}\label{inequn_Q_P_Q1}
(n-1)Q<P\leq Q<0
\end{align}
on $\overline{M}\times(0,T)$, and moreover,
\begin{align}\label{inequn_P-Q-2}
P(r,t)<Q(r,t)
\end{align}
for $(r,t)\in(0,r_0]\times(0,T)$. 

\end{thm}
\begin{proof}
By Lemma \ref{lem_P_Q-0}, $(\ref{inequn_pin1})$ holds on $\overline{M}\times[0,\epsilon]$ for some $\epsilon\in(0,T)$. Define the function
\begin{align*}
\xi(t)\triangleq \sup_{\overline{M}}\max\{Q(\cdot,t),\,(P-Q)(\cdot,t),\,((n-1)Q-P)(\cdot,t)\}
\end{align*}
for $T>t\geq0$. Then $\xi(t)$ is continuous on $[0,T)$. 

Assume, to the contrary, that there exists $t_0\geq \epsilon$, such that $(\ref{inequn_pin1})$ holds on $\overline{M}\times[0,t_0]$, while for each $\delta>0$, there exists $\tilde{t}\in(t_0,t_0+\delta)$ such that
\begin{align}\label{inequn_xi1}
\xi(\tilde{t})>0.
\end{align}
Since $P(0,t)=Q(0,t)$ for $t\geq0$, we have
\begin{align}
&\xi(t)=0,
\end{align}
for $t\in[0,t_0]$. Since $P-Q\leq 0$ on $\overline{M}\times[0,t_0]$, applying the strong maximum principle for $(\ref{equn_P-Qevol})$(see, for instance, Theorem 3 on Page 38 and Theorem 5 on Page 39 in \cite{Friedman}), we have that
\begin{align}\label{inequn_P-Q-large}
(P-Q)(r,t)<0
\end{align}
for $(r,t)\in(0,r_0)\times(0,t_0]$. If $Q(r_1,t_0)=0$ for some $r_1\in[0,r_0]$, then by $(\ref{inequn_pin1})$ on $\overline{M}\times[0,t_0]$, we have \begin{align*}
P(r_1,t_0)=Q(r_1,t_0)=(P-Q)(r_1,t_0)=0,
\end{align*}
and hence by $(\ref{inequn_P-Q-large})$, $r_1=r_0$ or $r_1=0$.  Therefore,
\begin{align}\label{inequn_Q-large}
Q(r,t_0)<0
\end{align}
for $r\in(0,r_0)$. Now if $Q(0,t_0)=0$, since
\begin{align*}
(P-Q)(0,t_0)=0,
\end{align*}
by continuity of $(P-Q)$, we have that $Q$ satisfies the inequality $(\ref{inequn_Qt})$ in a neighborhood $U$ of the point $(0,t_0)$ in $\overline{M}\times[0,t_0]$. Since $Q\leq 0$ on $U$ with $Q(0,t_0)=0$, by the strong maximum principle for $(\ref{inequn_Qt})$, we have $Q=0$ in $U$, contradicting with $(\ref{inequn_Q-large})$. Therefore,
\begin{align}\label{inequn_Q-large1}
Q(r,t_0)<0
\end{align}
for $r\in[0,r_0)$, and hence, by $(\ref{equn_Pn-1Qs})$, $(\ref{inequn_P-Q-large})$ and Lemma \ref{lem_Hp}, we have
\begin{align*}
(P-(n-1)Q)(r,t_0)\geq-(n-2)Q(0,t_0)>0,
\end{align*}
for $r\in[0,r_0]$. The same discussion yields
\begin{align}\label{inequn_P_P-Q_Q-1}
(n-1)Q<P\leq Q<0
\end{align}
for $(r,t)\,\in\,[0,r_0)\times(0,t_0]$. Therefore, again, by $(\ref{inequn_P-Q-large})$, $(\ref{equn_Pn-1Qs})$ and Lemma \ref{lem_Hp}, we get
\begin{align}\label{inequn_Pn-1Q2}
P-(n-1)Q>0
\end{align}
for $(r,t)\,\in\,[0,r_0]\times(0,t_0]$.

If $Q(r_0,t)=0$ for some $t\in(0,t_0]$, then by $(\ref{inequn_pin1})$ on $\overline{M}\times[0,t_0]$, we have
\begin{align*}
P(r_0,t)=(P-Q)(r_0,t)=(P-(n-1)Q)(r_0,t)=0,
\end{align*}
contradicting with $(\ref{inequn_Pn-1Q2})$. Therefore, we have
\begin{align}\label{inequn_Q-2}
Q<0
\end{align}
on $\overline{M}\times(0,t_0]$.

Now assume
\begin{align*}
(P-Q)(r_0,t)=0
\end{align*}
for some $t\in(0,t_0]$. Hence, by $(\ref{inequn_P_P-Q_Q-1})$, we obtain
\begin{align*}
\frac{\partial}{\partial n_g}(P-Q)(r_0,t)\geq0,
\end{align*}
contradicting with the first equality of $(\ref{equn_nP-Qb})$, by $(\ref{inequn_Pn-1Q2})$ and the fact $\eta'(t)\geq0$ for $t\geq0$.

In summary, we have that
\begin{align}
(n-1)Q<P\leq Q<0
\end{align}
on $\overline{M}\times(0,t_0]$, and moreover,
\begin{align}
P(r,t)<Q(r,t)
\end{align}
for $(r,t)\in(0,r_0]\times(0,t_0]$. Therefore, by continuity, for $\sigma\in(0,r_0)$ small, there exists $\delta>0$, such that
\begin{align}
(n-1)Q(r,t)<P(r,t)< Q(r,t)<0
\end{align}
for $(r,t)\in [\sigma,r_0]\times[t_0,t_0+\delta]$, and also,
\begin{align}\label{inequn_Pn-1QQ1}
(n-1)Q(r,t)-P(r,t)<0,\,\,\, Q(r,t)<0,
\end{align}
for $(r,t)\in [0,r_0]\times[t_0,t_0+\delta]$. Thus,
\begin{align*}
1-n-nb^{-2}-\frac{2}{n-2}(P-(n-1)Q)<0
\end{align*}
for $(r,t)\in(0,r_0]\times[0,t_0+\delta]$. Recall that $(P-Q)(0,t)=0$ for $t\geq0$. By the strong maximum principle for $(\ref{equn_P-Qevol})$ on $(0,\sigma)\times[t_0,t_0+\delta]$, we have
\begin{align}
(P-Q)(r,t)<0
\end{align}
for $(r,t)\in(0,\sigma)\times[t_0,t_0+\delta]$. 
 Therefore,
\begin{align}\label{equn_P-Q2}
(P-Q)(r,t)<0
\end{align}
for $(r,t)\in(0,r_0]\times(0,t_0+\delta]$.

By $(\ref{inequn_Pn-1QQ1})$, $(\ref{equn_P-Q2})$ and the fact $(P-Q)(0,t)=0$ for $t\geq0$, we have
\begin{align*}
\xi(t)=0
\end{align*}
for $t\in[t_0,t_0+\delta]$, contradicting with $(\ref{inequn_xi1})$. Therefore, we obtain $(\ref{inequn_Q_P_Q1})$ on $\overline{M}\times(0,T)$, and $(\ref{inequn_P-Q-2})$ for $(r,t)\in(0,r_0]\times(0,T)$. This completes the proof of the theorem.

\end{proof}

We now turn to the proof of the long-time existence of the solution to the initial-boundary value problem $(\ref{equn_NRF1})-(\ref{equn_Ninc2})$.
\begin{thm}\label{thm_existlongtime}
Let $n\geq3$. Assume that $\eta=\eta(t)\in C^k([0,\infty))$ is non-decreasing for $t\geq0$ and $\eta'(t)>0$ on $(0,t_0)$ for some $t_0>0$. Then there exists a unique solution to $(\ref{equn_NRF1})-(\ref{equn_Ninc2})$ on $\overline{M}\times[0,\infty)$, under the compatibility conditions $(\ref{equn_comptk})$ for some $k\geq2$.
\end{thm}
\begin{proof}
Applying the variable transformation $(\ref{equn_vartransf})$, Theorem 1.2 and Theorem 1.3 in \cite{Gi} yield: there exists a unique solution $g(t)\in C^{2k+\alpha,k+\frac{\alpha}{2}}(\overline{M}\times[0,\epsilon_0])$ to $(\ref{equn_NRF1})-(\ref{equn_Ninc2})$ for some $\epsilon_0>0$. Let $T$ be the maximal existence time of the solution $g(t)$.

Now assume $T<\infty$. For any $t_0\in(0,T)$, since $(\ref{inequn_pin1})$ holds on $\overline{M}\times[0,T)$, by the maximum principle for $(\ref{equn_P-Qevol})$, we have that there exists $t_1\in(0,t_0]$ such that
\begin{align}\label{equn_P-Qbi}
(P-Q)(r_0,t_1)=\inf_{\overline{M}\times[0,t_0]}(P-Q)<0,
\end{align}
and hence
\begin{align}
\frac{\partial}{\partial n_g}(P-Q)(r_0,t_1)\leq 0
\end{align}
where $n_g$ is the outer normal vector field of $\partial M$. Thus, by the first equality in $(\ref{equn_nP-Qb})$, we have
\begin{align}\label{inequn_P-Qb2}
(-P+Q)(r_0,t_1)\eta(t_1)-\frac{\eta(t_1)}{n-1}(P-(n-1)Q)(r_0,t_1)\leq\frac{n-2}{n-1}\eta'(t_1).
\end{align}
On the other hand, by $(\ref{equn_Pn-1Qs})$, $(\ref{inequn_P-Q-2})$ and Lemma \ref{lem_Hp}, we obtain that $(P-(n-1)Q)(r,t)>0$ is strictly increasing in $r$ for each $t>0$. The first equality in $(\ref{equn_Pn-1Qb1})$ yields
\begin{align}
P(r_0,t)-(n-1)Q(r_0,t)&=-(n-1)(n-2)[1+b_0^{-2}e^{2\int_0^tQ(r_0,\tau)d\tau}]+\frac{n-2}{n-1}\eta(t)^2\nonumber\\
&\label{inequn_Pn-1Qb2}\leq -(n-1)(n-2)+\frac{n-2}{n-1}\eta(t)^2
\end{align}
for $t\in(0,T)$, and hence,
\begin{align}\label{inequn_Pn-1Q3}
P(r,t)-(n-1)Q(r,t)\leq -(n-1)(n-2)+\frac{n-2}{n-1}\eta(t)^2
\end{align}
for $(r,t)\in[0,r_0]\times(0,T)$. By $(\ref{inequn_P-Qb2})$ and $(\ref{inequn_Pn-1Qb2})$, we get
\begin{align}
P(r_0,t_1)-Q(r_0,t_1)\geq n-2-\frac{n-2}{(n-1)^2}\eta(t_1)^2-\frac{n-2}{n-1}\frac{\eta'(t_1)}{\eta(t_1)},
\end{align}
and hence by $(\ref{equn_P-Qbi})$,
\begin{align}\label{inequn_P-Q3}
P(r,t)-Q(r,t)\geq n-2-\frac{n-2}{(n-1)^2}\eta(t_1)^2-\frac{n-2}{n-1}\frac{\eta'(t_1)}{\eta(t_1)}
\end{align}
for $(r,t)\in [0,r_0]\times[0,t_0]$. Combining $(\ref{inequn_Pn-1Q3})$ with $(\ref{inequn_P-Q3})$, we have
that
\begin{align}
 n(n-1)-\eta(t_0)^2-\frac{1}{n-1}\eta(t_1)^2-\frac{\eta'(t_1)}{\eta(t_1)}\leq (n-1)Q<P\leq Q<0
\end{align}
for $(r,t)\in[0,r_0]\times(0,t_0]$. Since $g(t)$ is rotationally symmetric for $t>0$, the second fundamental form $II_g$ of $\partial M$ satisfies
\begin{align}
II_g=\frac{\eta(t)b(r_0,t)^{2}}{n-1}g_{\mathbb{S}^{n-1}}.
\end{align}
In summary, we have that there exists $C=C(T)>0$ such that
\begin{align}
&\sup_{M\times[0,T)}|Rm_{g(t)}|_{g(t)}+\sup_{\partial M\times[0,T)}|II_{g(t)}|_{g(t)}\leq C.
\end{align}
By $(\ref{equn_metricpolar}),\,(\ref{equn_aintt0})$ and $(\ref{equn_bintt0})$, we obtain that $a(r,t)$ and $b(r,t)$ are increasing in $t$, and there exists a constant $C_1=C_1(T)>0$ such that
\begin{align}
b(r_0,0)\leq b(r_0,t)\leq C_1
\end{align}
for $t\in[0,T)$. Notice that $g^T(t)=b(r_0,t)^2g_{\mathbb{S}^{n-1}}$. By taking $\gamma(t)=g_{\mathbb{S}^{n-1}}$ for $t\in[0,T)$, one has that the boundary $\partial M$ is $\Lambda$-controlled in $(0,t_0]$ for any $t_0\in(0,T)$ (see Definition 3.1 in \cite{Gi1}). Hence, by Theorem 1.2 in \cite{Gi1}, for any $j=0,1,..,2k-2$, there exists a constant $C_2=C_2(T,k)>0$ such that
\begin{align}
&|\nabla^jRm_{g(t)}|_{g(t)}\leq C_2,\,\,\,\text{in}\,\,\overline{M}\times [0,T),\\
&|\nabla^{j+1}II_{g(t)}|_{g(t)}\leq C_2,\,\,\,\text{on}\,\,\partial M\times [0,T).
\end{align}
 By $(\ref{equn_aintt0})$ and $(\ref{equn_bintt0})$, the solution $g(t)$ can be extended to time $T$, and hence, by Theorem 1.2 and Theorem 1.3 in \cite{Gi}, the solution $g(t)$ exists on $[0,T+\delta]$ for some $\delta>0$, contradicting with the choice of $T$. Therefore, the solution $g(t)$ exists for all time $t>0$. This completes the proof of the theorem.

\end{proof}

\vskip0.2cm

\section{Locally uniform convergence of the solution to $(\ref{equn_NRF1})-(\ref{equn_Ninc2})$ in the interior of $\overline{M}$}\label{section_4}

Let $\overline{M}$ be of dimension $n\geq3$. In this section, we will show that both the coefficients $a(r,t)$ and $b(r,t)$ in the decomposition $(\ref{equn_metricpolar1})$ of the metric $g(t)$ and the curvatures $P(r,t)$ and $Q(r,t)$ are uniformly bounded on $[0,r_1]\times[0,\infty)$ for any $r_1\in(0,r_0)$; moreover, as $t\to\infty$, these functions converge locally uniformly in $r$ on $[0,r_0)$.

First, we show that $a(r,t)$, $b(r,t)$ and the distance function $s=\int_0^ra(x,t)dx$ are uniformly bounded on $[0,r_1]\times[0,\infty)$ for any $r_1\in(0,r_0)$.
\begin{lem}\label{lem_absBD}
Let $n\geq3$. Under the assumption in Theorem \ref{thm_existlongtime}, for each $r_1\in(0,r_0)$, there exists a constant $C=C(r_1)>0$ such that
\begin{align}
&\label{inequn_LBa}a(r,0)\leq a(r,t)\leq C,\\
&\label{inequn_LBb}b(r,0)\leq b(r,t)\leq Cr,\\
&\label{inequn_LBs}s(r,0)\leq s(r,t)\leq Cr,
\end{align}
for $(r,t)\in[0,r_1]\times[0,\infty)$.
\end{lem}
\begin{proof}
It follows from Theorem \ref{thm_PQn-1Q2} and Theorem \ref{thm_existlongtime} that $(\ref{inequn_Q_P_Q1})$ holds on $\overline{M}\times(0,\infty)$, and $(\ref{inequn_P-Q-2})$ holds for $(r,t)\in(0,r_0]\times(0,\infty)$. Thus, using the monotonicity of both $a$ and $b$ in $t$, and by $(\ref{equn_saformul})$, we have that the first inequality in each of the three chains $(\ref{inequn_LBa})-(\ref{inequn_LBs})$ holds.

Since $P(r,t)<0$ for $t>0$, by definition,
\begin{align*}
&b_{ss}''-b\geq0,
\end{align*}
and hence,
\begin{align}
(e^{s}b_s')_s'-(e^{s}b)_s'\geq0,
\end{align}
for $(r,t)\in[0,r_0]\times[0,\infty)$. Here $s=\int_0^ra(x,t)dx$, for $t\geq0$. Integrating this inequality with respect to $s$ and by the fact $b_s'(0,t)=1$ for any $t\geq0$, we obtain
\begin{align}\label{inequn_bsb1}
b_s'-b-e^{-s}\geq0
\end{align}
on $\overline{M}\times[0,\infty)$, and hence by the fact $b(0,t)=0$ for $t\geq0$,
\begin{align}\label{inequn_bsinhslbd}
b\geq \sinh(s)
\end{align}
on $\overline{M}$ for any $t\geq0$. Thus, by $(\ref{inequn_bsb1})$,
\begin{align}
b_s'\geq \cosh(s)
\end{align}
on $\overline{M}$ for any $t\geq0$.

On the other hand, by $(\ref{equn_aintt0})$, $(\ref{equn_bintt0})$ and $(\ref{inequn_Q_P_Q1})$,
\begin{align}\label{inequn_bn-1ab2}
(n-1)\ln(\frac{b(r,t)}{b(r,0)})&=-(n-1)\int_0^tQ(r,\tau)d\tau>-\int_0^tP(r,\tau)d\tau
= \ln(\frac{a(r,t)}{a(r,0)})\\
&\geq -\int_0^tQ(r,\tau)d\tau =\ln(\frac{b(r,t)}{b(r,0)}) \nonumber
\end{align}
for $(r,t)\in(0,r_0]\times(0,\infty)$, and hence, by Lemma \ref{lem_Hp}, for any $0< r_1<r_2\leq r_0$,
\begin{align*}
\ln(\frac{a(r_2,t)}{a(r_2,0)})\geq \ln(\frac{b(r_2,t)}{b(r_2,0)})&\geq \ln(\frac{b(r_1,t)}{b(r_2,0)})
=\ln(\frac{b(r_1,t)}{b(r_1,0)})+\ln(\frac{b(r_1,0)}{b(r_2,0)})\\
&\geq \frac{1}{n-1}\ln(\frac{a(r_1,t)}{a(r_1,0)})+\ln(\frac{b(r_1,0)}{b(r_2,0)})
\end{align*}
for $t>0$. Thus,
\begin{align}\label{inequn_a2a1}
\frac{a(r_2,t)}{a(r_2,0)}\geq\,\frac{b(r_1,0)}{b(r_2,0)}\,\big(\frac{a(r_1,t)}{a(r_1,0)}\big)^{\frac{1}{n-1}}
\end{align}
for any $0\leq r_1<r_2\leq r_0$ and $t>0$.

Now, by definition of $s$, $(\ref{inequn_bn-1ab2})$ and $(\ref{inequn_bsinhslbd})$,
\begin{align*}
\frac{\partial s}{\partial r}=a(r,t)\geq \frac{a(r,0)}{b(r,0)} b(r,t) \geq \frac{a(r,0)}{b(r,0)} \sinh(s),
\end{align*}
and hence,
\begin{align}\label{inequn_sabr1}
\frac{ds}{\sinh(s)}\geq \frac{a(r,0)}{b(r,0)} dr
\end{align}
for $t>0$. For any $0< r_1\leq r_2\leq r_0$, integrating $(\ref{inequn_sabr1})$ over $r\in[r_1,r_2]$ yields
\begin{align}
\ln(\frac{e^{s(r_2,t)}-1}{e^{s(r_2,t)}+1})-\ln(\frac{e^{s(r_1,t)}-1}{e^{s(r_1,t)}+1})\geq \int_{r_1}^{r_2}a(r,0)b^{-1}(r,0)dr
\end{align}
for $t>0$. Thus,
\begin{align}
1>\frac{e^{s(r_2,t)}-1}{e^{s(r_2,t)}+1}&\geq\,\frac{e^{s(r_1,t)}-1}{e^{s(r_1,t)}+1}\,e^{\int_{r_1}^{r_2}a(r,0)b^{-1}(r,0)dr}\\
&=\,\frac{e^{s(r_1,t)}-1}{e^{s(r_1,t)}+1}\,\frac{e^{r_2}-1}{e^{r_2}+1}\,\frac{e^{r_1}+1}{e^{r_1}-1}
\end{align}
for $t>0$. Now take $r_2=r_0$, and hence
\begin{align}
&\frac{e^{s(r,t)}-1}{e^{s(r,t)}+1}<\frac{e^{r_0}+1}{e^{r_0}-1}\,\frac{e^{r}-1}{e^{r}+1},\\
&\label{inequn_supbd2}s(r,t)<\ln(\frac{e^{r_0+r}-1}{e^{r_0}-e^r})<\infty,
\end{align}
for any $(r,t)\in(0,r_0)\times(0,\infty)$, which proves the second inequality in $(\ref{inequn_LBs})$. Thus,
\begin{align*}
a(0,t)=\lim_{r\to0_+}\frac{s(r,t)}{r}\leq \frac{e^{r_0}+1}{e^{r_0}-1},
\end{align*}
for any $t>0$.

Now, we show that for any $r_1\in(0,r_0)$, $a(r,t)$ and $b(r,t)$ are uniformly bounded on $[0,r_1]\times[0,\infty)$. Assume for contradiction that there exists a sequence of points $\{(\bar{r}_j,t_j)\}_{j=1}^\infty$ such that $a(\bar{r}_j,t_j)\to\infty$, $\epsilon\leq \bar{r}_j\leq r_1$ and $t_j\to\infty$, for some $\epsilon\in(0,r_1)$. Then for any $r\in[r_1,\frac{r_1+r_0}{2}]$, taking $r_1=\bar{r}_j$ and $r_2=r$ in $(\ref{inequn_a2a1})$, one gets
\begin{align}
a(r,t_j)\geq\,\frac{a(r,0)b(\bar{r}_j,0)}{b(r,0)}\,\big(\frac{a(\bar{r}_j,t_j)}{a(\bar{r}_j,0)}\big)^{\frac{1}{n-1}},
\end{align}
for $r\in[r_1,\frac{r_1+r_0}{2}]$, and hence,
\begin{align}
s(\frac{r_1+r_0}{2},t_j)-s(r_1,t_j)=\int_{r_1}^{\frac{r_1+r_0}{2}}a(r,t_j)dr\to\infty
\end{align}
as $t_j\to\infty$, contradicting to $(\ref{inequn_supbd2})$. Therefore, there exists $C=C(r_1,\epsilon)>0$ such that
\begin{align}
a(r,t)\leq C
\end{align}
for $(r,t)\in[\epsilon,r_1]\times[0,\infty)$, and hence, by $(\ref{inequn_bn-1ab2})$,
\begin{align}\label{inequn_bupbb1}
b(r,t)\leq \frac{C}{a(r,0)}b(r,0)=C\sinh(r)
\end{align}
for $(r,t)\in[\epsilon,r_1]\times[0,\infty)$. Since $b_r'(r,t)>0$, we have $b(r,t)\leq C\sinh(r_1)$ for $(r,t)\in[0,r_1]\times[0,\infty)$.

We then show that $P-(n-1)Q$ is uniformly bounded from above on $[0,r_1]\times[0,\infty)$ for any $r_1\in(0,r_0)$.

Assume, for contradiction, that there exists a sequence$\{\bar{r}_j\}_{j=1}^\infty$ on $[0,r_1]$ and $\{t_j\}_{j=1}^\infty\subseteq [0,\infty)$ with $t_j\to\infty$ as $j\to\infty$ such that
\begin{align*}
(P-(n-1)Q)(\bar{r}_j,t_j)\to \infty
\end{align*}
as $j\to\infty$. Since $P-(n-1)Q$ is increasing in $r$ on $[0,r_0]$ for any $t>0$, we obtain
\begin{align}\label{equn_Pn-1Qblup}
(P-(n-1)Q)(r,t_j)=-(n-1)(n-2)[b(r,t_j)^{-2}(1-(b_s'(r,t_j))^2)+1]\to +\infty
\end{align}
uniformly in $r$ on $[r_1,\frac{r_1+r_0}{2}]$ as $j\to\infty$. Since $b(r,t)\geq b(r,0)$ is uniformly bounded from below on $[r_1,\frac{r_1+r_0}{2}]\times[0,\infty)$, we obtain
\begin{align*}
b_s'(r,t_j)\to \infty
\end{align*}
uniformly in $r$ on $[r_1,\frac{r_1+r_0}{2}]$, as $j\to\infty$. Therefore,
\begin{align*}
b(\frac{r_1+r_0}{2},t_j)-b(r_1,t_j)=\int_{s(r_1,t_j)}^{s(\frac{r_1+r_0}{2},t_j)}b_s'(r,t_j)ds\to\infty,
\end{align*}
as $j\to\infty$, contradicting with $(\ref{inequn_bupbb1})$. 
 Therefore, for any $r_1\in(0,r_0)$, there exists a constant $C_1=C_1(r_1)>0$ such that
 \begin{align}
 &(P-(n-1)Q)(r,t)\leq C_1,\\
 &\label{inequn_b'supbd}b_s'(r,t)\leq C_1,
 \end{align}
for $(r,t)\in[0,r_1]\times[0,\infty)$, where for $(\ref{inequn_b'supbd})$ we have used the first equality in $(\ref{equn_Pn-1Qblup})$ again. By the definition of $P$, since $P<0$ on $\overline{M}\times(0,\infty)$ and $b_s'(0,t)=1$ for $t\geq0$, we obtain that $b_s'$ is increasing in $s$ and
\begin{align}
1\leq b_s'(r,t)\leq C_1
\end{align}
for $(r,t)\in[0,r_1]\times[0,\infty)$, and hence,
\begin{align*}
s\leq b(r,t)=\int_0^sb_s'(r,t)ds \leq C_1s\leq C_1C_2r
\end{align*}
with some constant $C_2=C_2(r_1)>0$, for $(r,t)\in[0,r_1]\times[0,\infty)$, where the last inequality is by $(\ref{inequn_supbd2})$. Thus, by $(\ref{inequn_bn-1ab2})$, we have
\begin{align*}
a(r,t)\leq (\frac{b(r,t)}{b(r,0)})^{n-1} a(r,0)\leq C'
\end{align*}
with some constant $C'=C'(r_1)>0$, for $(r,t)\in [0,r_1]\times[0,\infty)$. This completes the proof of the lemma.

\end{proof}

Next, for each $r_1\in(0,r_0)$, we prove that $P(r,t)$ and $Q(r,t)$ are uniformly bounded on $[0,r_1]\times[0,\infty)$, and moreover, $P(r,t)$ and $Q(r,t)$ converge to $0$ uniformly for $r\in[0,r_1]$ as $t\to\infty$.

\begin{thm}\label{thm_QPn-1Qunibd2}
Let $n\geq3$. Assume that $\eta=\eta(t)\in C^k([0,\infty))$ is non-decreasing for $t\geq0$ and $\eta'(t)>0$ on $(0,t_0)$ for some $t_0>0$. Assume the compatibility conditions $(\ref{equn_comptk})$ hold for some $k\geq2$. Then for any $r_1\in(0,r_0)$, there exists a constant $C=C(r_1)>0$ such that
\begin{align}\label{inequn_Qlb2}
-C\leq (n-1)Q<P\leq Q<0
\end{align}
on $[0,r_1]\times(0,\infty)$. Moreover, for each $m\in \mathbb{N}$ and $r_1\in(0,r_0)$, there exists $C=C(n,m,r_1)>0$, such that
\begin{align*}
|\nabla^mRm_{g(t)}|_{g(t)}\leq C
\end{align*}
on $\overline{B_{r_1}}(0)\times[0,\infty)$.
\end{thm}
\begin{proof}

First, we show that for any $r_1\in(0,r_0)$, $P(r,t)$ and $Q(r,t)$ are uniformly bounded from below on $[0,r_1]\times[0,\infty)$.

By $(\ref{equn_envP})$ and $(\ref{equn_envQ})$, we get
\begin{align}\label{equn_Pn-1Qt}
\partial_t(P+(n-1)Q)=\Delta_g(P+(n-1)Q)+2(1-n)(P+(n-1)Q)+2(P^2+(n-1)Q^2)
\end{align}
on $\overline{M}\times[0,\infty)$, with
\begin{align*}
\frac{2}{n}(P+(n-1)Q)^2\leq 2(P^2+(n-1)Q^2)\leq 2(P+(n-1)Q)^2,
\end{align*}
by $(\ref{inequn_Q_P_Q1})$, and hence,
\begin{align}\label{equn_Pn-1Qt2}
\partial_t(P+(n-1)Q)=\Delta_g(P+(n-1)Q)+2((1-n)+f)(P+(n-1)Q)
\end{align}
on $\overline{M}\times[0,\infty)$, with
\begin{align*}
f= \frac{(P^2+(n-1)Q^2)}{(P+(n-1)Q)}<0,
\end{align*}
continuous on $\overline{M}\times(0,\infty)$. By the maximum principle for $(\ref{equn_Pn-1Qt2})$, once $P+(n-1)Q$ is uniformly bounded on $\{r_1\}\times[0,\infty)$, $P+(n-1)Q$ is uniformly bounded on $[0,r_1]\times[0,\infty)$.

Assume, to the contrary, that there exists a sequence $t_j\nearrow+\infty$ as $j\to\infty$, such that $(P+(n-1)Q)(r_1,t_j)\to -\infty$. Then by $(\ref{inequn_pin1})$,
\begin{align}
P(r_1,t_j)\to-\infty,\,\,\,Q(r_1,t_j)\to-\infty
\end{align}
as $j\to\infty$. Since $P-(n-1)Q$ is uniformly bounded on $[0,r_1]\times[0,\infty)$, we obtain
\begin{align}\label{inequn_P-Q2blup}
P(r_1,t_j)-Q(r_1,t_j)\to-\infty
\end{align}
as $j\to\infty$. Let $\delta=\frac{1}{2}\min\{r_1, r_0-r_1\}$. On $[r_1-\delta, r_1+\delta]\times(0,\infty)$, we rewrite $(\ref{equn_P-Qevol})$ as
\begin{align*}
\partial_t(P-Q)=&a^{-2}\partial_r^2(P-Q)+[-a^{-3}\partial_ra\,+(n-1)a^{-1}b^{-1}b_s']\,\partial_r(P-Q)\\
&+2[1-n-nb^{-2}-\frac{2}{n-2}(P-(n-1)Q)](P-Q).
\end{align*}
Now we define the continuous functions
\begin{align*}
&f_1(r,t)=-a^{-3}\partial_ra\,+(n-1)a^{-1}b^{-1}b_s',\\
&f_2(r,t)=2[1-n-nb^{-2}-\frac{2}{n-2}(P-(n-1)Q)],
\end{align*}
on $[r_1-\delta, r_1+\delta]\times(0,\infty)$. Therefore,
\begin{align}\label{equn_P-Qt2}
\partial_t(P-Q)=a^{-2}\partial_r^2(P-Q)+f_1(r,t)\,\partial_r(P-Q)+f_2(r,t)(P-Q).
\end{align}

{\bf Claim:} there exists $C=C(r_1)>0$ such that
\begin{align}
|\partial_ra|\leq C
\end{align}
on $[r_1-\delta, r_1+\delta]\times(0,\infty)$.

Once the claim holds, since $P-(n-1)Q$ and $b_s'$ are uniformly bounded on $[r_1-\delta, r_1+\delta]\times(0,\infty)$ by Lemma \ref{lem_absBD} and its proof, there exists $C_2>0$ such that
\begin{align*}
|f_1(r,t)|+|f_2(r,t)|\leq C_2
\end{align*}
on $[r_1-\delta, r_1+\delta]\times(0,\infty)$. Hence, it follows from the parabolic Harnack inequality for $(\ref{equn_P-Qt2})$ that there exists $C>0$ independent of $j\in \mathbb{N}$ such that
\begin{align}
(Q-P)(r,t)\geq C(Q-P)(r_1,t_j)>0
\end{align}
for $(r,t)\in U_{j,\frac{\delta}{2}}=\{(r,t):\,|r-r_1|\leq \frac{\delta}{2},\,\,t_j+\frac{\delta^2}{8}\leq t\leq t_j+\frac{\delta^2}{4}\}$ and any $t_j\geq \delta^2$. Now we choose a subsequence $\{t_{j_k}\}_{k=1}^\infty\subseteq \{t_j\}_{j=1}^\infty$ such that $t_{j_1}>\delta^2$ and $t_{j_{k+1}}>t_{j_k}+\delta^2$, for $k\geq1$. Since $Q<0$ for $t>0$, we obtain
\begin{align}
-P(r,t)\geq C(Q-P)(r_1,t_{j_k})>0
\end{align}
for $(r,t)\in U_{j_k,\frac{\delta}{2}}$ and any $k\in \mathbb{N}$, and hence by $(\ref{inequn_P-Q2blup})$,
\begin{align*}
a(r_1,t_{j_m})=a(r_1,0)e^{\int_0^{t_{j_m}}-P(r_1,\tau)d\tau}&\geq a(r_1,0)e^{C\sum_{k=1}^{m-1}\int_{t_{j_k}+\frac{\delta^2}{8}}^{t_{j_k}+\frac{\delta^2}{4}}(Q-P)(r_1,t_{j_k})d\tau}\\
&=a(r_1,0)e^{\frac{C\delta^2}{8}\sum_{k=1}^{m-1}(Q-P)(r_1,t_{j_k})}\to\infty,
\end{align*}
as $m\to\infty$, contradicting with Lemma \ref{lem_absBD}. This proves the uniform boundedness of $P+(n-1)Q$, and hence $P$ and $Q$, on $[0,r_1]\times[0,\infty)$.

Now, it remains to prove the claim.

It follows by direct calculation that
\begin{align*}
&\partial_t(a_r')=\partial_r(\partial_ta)=-\partial_r(Pa)=-P\partial_ra-a\partial_rP,\\
&\partial_t(b_r')=\partial_r(\partial_tb)=-\partial_r(Qb)=-Q\partial_rb-b\partial_rQ,
\end{align*}
and hence,
\begin{align*}
\partial_t(a^{-1}a_r'-(n-1)b^{-1}b_r')=\partial_r((n-1)Q-P)&=a\partial_s((n-1)Q-P)\\
&=2(n-1)ab^{-1}b_s'\,(P-Q)
\end{align*}
where we have used $(\ref{equn_NRFa})$, $(\ref{equn_NRFb})$, $(\ref{equn_Pn-1Qs})$ and the definition of $P$ and $Q$. Integrating both sides of this equation with respect to $t$, we obtain
\begin{align}\label{equn_arbr}
\,\,\,&a(r,t)^{-1}a_r'(r,t)-(n-1)b(r,t)^{-1}b_r'(r,t)\\
=\,\,\,&a(r,0)^{-1}a_r'(r,0)-(n-1)b(r,0)^{-1}b_r'(r,0)\nonumber\\
+2&(n-1)\int_0^ta(r,\tau)b(r,\tau)^{-1}b_s'(r,\tau)P(r,\tau)d\tau-2(n-1)\int_0^ta(r,\tau)b(r,\tau)^{-1}b_s'(r,\tau)Q(r,\tau)d\tau,\nonumber
\end{align}
on $[r_1-\delta, r_1+\delta]\times(0,\infty)$. On the other hand, $a,\,a^{-1},\,b,\,b^{-1},\,b_s'$, and hence $b_r'=ab_s'$, are uniformly bounded on $[r_1-\delta, r_1+\delta]\times(0,\infty)$, and moreover, we have $P<Q<0$ on $[r_1-\delta, r_1+\delta]\times(0,\infty)$ and the integrals
\begin{align*}
&\int_0^t-P(r,\tau)d\tau=\ln(\frac{a(r,t)}{a(r,0)}),\\
&\int_0^t-Q(r,\tau)d\tau=\ln(\frac{b(r,t)}{b(r,0)})
\end{align*}
are uniformly bounded on $[r_1-\delta, r_1+\delta]\times(0,\infty)$. Therefore, by $(\ref{equn_arbr})$, there exists a constant $C_3=C_3(r_1)>0$ such that
\begin{align}
|a_r'(r,t)|\leq C_3
\end{align}
on $[r_1-\delta, r_1+\delta]\times(0,\infty)$. This proves the {\bf Claim}. Therefore, $(\ref{inequn_Qlb2})$ holds. By Shi's interior estimates in \cite{Shi} on the solution to the Ricci flow, for each $j\geq0$ and $r_1\in(0,r_0)$, there exists $C=C(n,j,r_1)>0$ such that
\begin{align*}
|\nabla^jRm_{g(t)}|_{g(t)}\leq C
\end{align*}
on $\overline{B_{r_1}}(0)\times[0,\infty)$. This completes the proof of the theorem.

\end{proof}

\begin{thm}
Let $n\geq3$. Assume that $\eta=\eta(t)\in C^k([0,\infty))$ is non-decreasing for $t\geq0$ and $\eta'(t)>0$ on $(0,t_0)$ for some $t_0>0$. Assume the compatibility conditions $(\ref{equn_comptk})$ hold for some $k\geq2$. Then $P$ and $Q$ converge to zero uniformly in $r$ on $[0,r_1]$ as $t\to\infty$, for any $r_1\in(0,r_0)$. Moreover, $g(t)$ converges locally uniformly in the interior of $\overline{M}$ to a rotationally symmetric locally hyperbolic metric $g_{\infty}$, as $t\to\infty$.
\end{thm}
\begin{proof}
The proof of the theorem is similar to the second part of the proof of Theorem \ref{thm_QPn-1Qunibd2}. By the proof of Theorem \ref{thm_QPn-1Qunibd2}, for each $r_1\in(0,r_0)$, $P, \,Q$, $b_s'$, and hence the function $f$ in $(\ref{equn_Pn-1Qt2})$, are uniformly bounded on $[0,r_1]\times[0,\infty)$, and moreover, by the claim in the proof of Theorem \ref{thm_QPn-1Qunibd2}, $|a_r(r,t)|$ is uniformly bounded on $[r_1-\delta,r_1+\delta]\times[0,\infty)$.

Now assume there exists a constant $\epsilon>0$ and an increasing sequence $\{t_j\}_{j=1}^\infty$ such that $t_j\to\infty$ and $(P+(n-1)Q)(r_1,t_j)<-\epsilon$. 
We rewrite $(\ref{equn_Pn-1Qt2})$ as
\begin{align*}
&\partial_t(P+(n-1)Q)-a^{-2}\partial_r^2(P+(n-1)Q)\\
=\,\,&f_3(r,t)\partial_r(P+(n-1)Q)+f_4(r,t)(P+(n-1)Q)
\end{align*}
on $[r_1-\delta,r_1+\delta]\times[0,\infty)$, where $f_3=-a^{-3}\partial_ra\,+(n-1)a^{-1}b^{-1}b_s'$ and $f_4=2((1-n)+f)$ are uniformly bounded on $[r_1-\delta,r_1+\delta]\times[0,\infty)$. Recall that $P+(n-1)Q\leq0$ on $\overline{M}\times[0,\infty)$. Thus it follows from the parabolic Harnack inequality for this equation that there exists $C>0$ independent of $j\in \mathbb{N}$ such that
\begin{align}
(P+(n-1)Q)(r,t)\leq C\,(P+(n-1)Q)(r_1,t_j)\leq - C\epsilon,
\end{align}
and hence by $(\ref{inequn_pin1})$,
\begin{align}
P(r,t)\leq \frac{-C\,\epsilon}{n},
\end{align}
for $(r,t)\in U_{j,\frac{\delta}{2}}=\{(r,t):\,|r-r_1|\leq \frac{\delta}{2},\,\,t_j+\frac{\delta^2}{8}\leq t\leq t_j+\frac{\delta^2}{4}\}$ and any $t_j\geq \delta^2$. Now we choose a subsequence $\{t_{j_k}\}_{k=1}^\infty\subseteq \{t_j\}_{j=1}^\infty$ such that $t_{j_1}>\delta^2$ and $t_{j_{k+1}}>t_{j_k}+\delta^2$, for $k\geq1$. Therefore,
\begin{align*}
a(r_1,t_{j_m})&=a(r_1,0)e^{-\int_0^{t_{j_m}}P(r_1,\tau)d\tau}\\
&\geq a(r_1,0)e^{\sum_{k=1}^{m-1}\int_{t_{j_k}+\frac{\delta^2}{8}}^{t_{j_k}+\frac{\delta^2}{4}}\frac{C\,\epsilon}{n}d\tau}\\
&=a(r_1,0)e^{\frac{\delta^2C\epsilon}{8n}\,(m-1)}\to\infty,
\end{align*}
as $m\to\infty$, contradicting with Lemma \ref{lem_absBD}. Therefore,
\begin{align*}
(P+(n-1)Q)(r_1,t)\to 0
\end{align*}
as $t\to\infty$. Since by Theorem \ref{thm_QPn-1Qunibd2}, $P$ and $Q$ are uniformly bounded in $C^k$-norm on $\overline{B_{r_1+\delta}}(0)\times[0,\infty)$ for any $k\geq0$, then by the Harnack inequality for $(\ref{equn_Pn-1Qt2})$,
\begin{align*}
(P+(n-1)Q)\to 0
\end{align*}
uniformly on $\overline{B_{r_1}}(0)$, as $t\to\infty$. Again, by Shi's interior estimates on the Ricci flow,
\begin{align*}
(P+(n-1)Q)\to 0
\end{align*}
uniformly in $C^k$-norm on $\overline{B_{r_1}}(0)$ for any $k\geq0$, as $t\to\infty$. Moreover, $a(r,t)$ and $b(r,t)$ are increasing in $t$ and converge to $a_{\infty}(r)$ and $b_{\infty}(r)$ locally uniformly in $C^k$-norm for any $k\geq0$ on $B_{r_0}(0)$, by $(\ref{equn_aintt0})$, $(\ref{equn_bintt0})$ and Lemma \ref{lem_absBD}. This completes the proof of the theorem.

\end{proof}

\vskip0.2cm

\section{Completeness of the limit metric}\label{section_5}
Let $\overline{M}$ be of dimension $n\geq3$. In this section, we will prove Theorem \ref{thm_convern345}, i.e., the solution $g(t)$ converges locally uniformly to the complete hyperbolic metric on $B_{r_0}(0)$. To this end, it suffices to prove that the volume of $B_{r_0}(0)$ under the metric $g_{\infty}$ is infinity, as $g_{\infty}=a_{\infty}(r)^2dr^2+b_{\infty}(r)^2g_{\mathbb{S}^{n-1}}$ is rotationally symmetric and locally hyperbolic.

\begin{proof}[Proof of Theorem \ref{thm_convern345}.]
It suffices to prove that the volume of $B_{r_0}(0)$ is infinity under the metric $g_{\infty}$. Since $a(r,t)$ and $b(r,t)$ are increasing in $t$, it remains to prove that
\begin{align}
\text{Vol}_{g(t)}(B_{r_0}(0))=\omega_n\int_0^{r_0}a(r,t)b(r,t)^{n-1}dr\to\infty\,
\end{align}
as $t\to\infty$, where $\omega_n$ is the area of the unit sphere $\mathbb{S}^{n-1}$ in $\mathbb{R}^n$.

Let $R_g$ be the scalar curvature of $g$, and define
\begin{align*}
R_g^\circ=R_g+n(n-1)=P+(n-1)Q.
\end{align*}
By $(\ref{equn_NRF1})$,
\begin{align*}
\frac{d}{dt} \text{Vol}_{g(t)}(B_{r_0}(0))&=\int_{B_{r_0}(0)}-R_{g(t)}^\circ dV_{g(t)}\\
&=\int_{B_{r_0}(0)}-(P+(n-1)Q) dV_{g(t)}.
\end{align*}
By $(\ref{equn_Pn-1Qt})$,
\begin{align*}
\partial_tR_{g(t)}^\circ=\Delta_gR_{g(t)}^\circ+2(1-n)R_{g(t)}^\circ+2(P^2+(n-1)Q^2)
\end{align*}
on $\overline{B_{r_0}}\times[0,\infty)$. By $(\ref{equn_nPn-1Qb})$ and $(\ref{equn_DtHb})$,
\begin{align*}
\frac{\partial}{\partial n_g}R_{g(t)}^\circ=-2\eta'(t)+2Q(r_0,t)\eta(t),
\end{align*}
on $\partial B_{r_0}(0)\times[0,\infty)$, where $n_{g}$ is the outer normal vector field of $\partial B_{r_0}(0)$ with respect to $g(t)$. Therefore,
\begin{align*}
&\frac{d}{dt}\int_{B_{r_0}(0)}R_{g(t)}^\circ dV_{g(t)}\\
=&\int_{B_{r_0}(0)}[-(R_g^\circ)^2+\Delta_gR_{g(t)}^\circ+2(1-n)R_{g(t)}^\circ+2(P^2+(n-1)Q^2)]dV_{g(t)}\\
=&\int_{B_{r_0}(0)}[-(P+(n-1)Q)^2+2(P^2+(n-1)Q^2)+2(1-n)R_{g(t)}^\circ]dV_g+\int_{\partial B_{r_0}(0)}\frac{\partial}{\partial n_g}R_{g(t)}^\circ dS_{g(t)}\\
=&\int_{B_{r_0}(0)}[P^2-2(n-1)PQ-(n-1)(n-3)Q^2+2(1-n)R_{g(t)}^\circ]dV_g\\
&+\int_{\partial B_{r_0}(0)}-2\eta'(t)+2Q(r_0,t)\eta(t) dS_{g(t)}.
\end{align*}
By $(\ref{inequn_pin1})$, we obtain $Q(r_0,t)<0$ for $t>0$, and
\begin{align*}
P^2-2(n-1)PQ\leq -(n-1)PQ<0
\end{align*}
on $\overline{B_{r_0}}(0)\times(0,\infty)$. On the other hand, by definition,
\begin{align*}
Q(r_0,t)&=\frac{1}{n-1}P(r_0,t)+(n-2)(b(r_0,t)^{-2}-\frac{1}{(n-1)^2}\eta(t)^2+1)\\
&\leq (n-2)(b(r_0,0)^{-2}-\frac{1}{(n-1)^2}\eta(t)^2+1)\\
&=-\frac{(n-2)}{(n-1)^2}(\eta(t)^2-\eta(0)^2)<0
\end{align*}
for $t>0$. Let
\begin{align*}
\varphi(t)=\int_{\partial B_{r_0}(0)}-\frac{2(n-2)}{(n-1)^2}(\eta(t)^2-\eta(0)^2)\eta(t) dS_{g(t)},
\end{align*}
for $t\geq0$. Then $\varphi\in C^{k-1}([0,\infty))$ and $\varphi(t)<0$ for $t>0$, and moreover, there exists $\epsilon>0$ and $N>0$, such that $\varphi(t)\leq - \epsilon$ for any $t\geq N$. Now, let $f(t)=\int_{B_{r_0}(0)}R_{g(t)}^\circ dV_{g(t)}$. Therefore,
\begin{align*}
\frac{df}{dt}\leq 2(1-n)f+\varphi(t)
\end{align*}
for $t\geq0$, and hence,
\begin{align*}
f(t)\leq e^{-2(n-1)t}\int_0^te^{2(n-1)\tau}\varphi(\tau)d\tau
\end{align*}
for $t\geq0$. Therefore, there exists $\epsilon_1>0$ and $N_1>0$ such that
\begin{align*}
f(t)\leq - \epsilon_1
\end{align*}
for $t\geq N_1$. Thus,
\begin{align*}
\text{Vol}_{g(t)}(B_{r_0}(0))&=\text{Vol}_{g(N_1)}(B_{r_0}(0))+\int_{N_1}^t-f(\tau)d\tau\\
&\geq \text{Vol}_{g(N_1)}(B_{r_0}(0))+\epsilon_1(t-N_1)
\end{align*}
for $t\geq N_1$, and hence,
\begin{align*}
\text{Vol}_{g(t)}(B_{r_0}(0))\to\infty
\end{align*}
as $t\to\infty$. This completes the proof of the theorem.

\end{proof}

\begin{Remark}
If $\eta(t)$ is uniformly bounded for $t\in[0,\infty)$, then by the monotonicity of $P-(n-1)Q$ in $r$ and $(\ref{equn_Pn-1Qb})$, $P-(n-1)Q$ is uniformly bounded on $\overline{B_{r_0}}(0)\times[0,\infty)$; while if $\eta(t)\to+\infty$ as $t\to\infty$, $(P-(n-1)Q)(r_0,t)\to+\infty$, and hence, $P(r_0,t)\leq Q(r_0,t)\to-\infty$,  as $t\to\infty$.
\end{Remark}

\vskip0.2cm

\section{The special case of dimension two}\label{section_6}

When $n=2$, assume the boundary data $\eta$ in $(\ref{equn_Nbdc1})$ satisfies that $\eta'(t)\geq0$ for $t\geq0$ and $\eta'(t)>0$ for $t\in(0,\epsilon)$ with some constant $\epsilon>0$. Short-time existence of the solution $g(t)$ to $(\ref{equn_NRF1d2})-(\ref{equn_Ninc2d2})$ was addressed in the preliminaries. Let $T>0$ be the largest existence time of the solution $g(t)$ to the initial-boundary problem $(\ref{equn_NRF1d2})-(\ref{equn_Ninc2d2})$.

Since $P=0$ at $t=0$, it follows from the continuity of $P$ that $P-1<0$ on $\overline{B_{r_0}}(0)\times[0,\epsilon_0)$ for some small $\epsilon_0>0$. By the maximum principle for $(\ref{equn_n2Pt})$, $P$ cannot attain its positive maximum or negative minimum in $B_{r_0}(0)\times(0,\epsilon_1]$ for any $\epsilon_1\in(0,\epsilon_0)$. 
We first show that
\begin{align}\label{inequn_P-0}
P(r,t)\leq0
\end{align}
on $\overline{B_{r_0}}(0)\times[0,T)$. In comparison to \cite{Li3}, the new contribution here for two-dimensional cases is the following theorem on curvature sign preservation:

\begin{thm}
For $n=2$, assume the boundary data $\eta$ in $(\ref{equn_Nbdc1})$ with $n=2$ satisfies that $\eta'(t)\geq0$ for $t\geq0$. Then $(\ref{inequn_P-0})$ holds on $\overline{B_{r_0}}(0)\times[0,T)$. If moreover, $\eta'(t)>0$ for $t\in(0,\epsilon)$ with some constant $\epsilon>0$, then
\begin{align}
P<0
\end{align}
on $\overline{M}\times(0,T)$.
\end{thm}
\begin{proof}
Assume the contrary, i.e., there exists $t_0\in[0,T)$, such that $(\ref{inequn_P-0})$ holds for $t\in [0,t_0]$ and for any $\delta>0$, there exists $t_1\in(t_0,t_0+\delta)$ such that
\begin{align*}
P(r_0,t_1)>0,
\end{align*}
where we have applied the maximum principle for $(\ref{equn_n2Pt})$, continuity of $P-1$, and the fact $P-1<0$ on $\overline{B_{r_0}}(0)\times[0,t_0]$.

Now we define a rotationally symmetric function $\tau(x)\in C^2(\overline{M})$ such that $\tau(x)=\tau(r(x))\geq \frac{1}{2}$ for $x\in\overline{M}\subseteq \mathbb{R}^2$ and $\tau(x)=1-s(r_0,t_0)+s(r(x),t_0)$ in a small neighborhood of $\partial M$, where $r(x)=|x|$ in the Euclidean space $\mathbb{R}^2$. Thus, $\tau(x)=1$ for $x\in \partial M$, and
\begin{align}\label{equn_taunb}
\frac{\partial}{\partial n_{g(t_0)}}\tau=1
\end{align}
on $\partial B_{r_0}(0)$, where $n_{g(t_0)}$ is the outer normal vector field of $\partial B_{r_0}(0)$ with respect to $g(t_0)$. Let $\theta=e^{-at}Pe^{-N\tau(x)}$, with $a,\,N>0$ two large constants to be determined. Thus for any $\delta>0$, by $(\ref{equn_n2Pt})$, we obtain
\begin{align}\label{equn_thetat}
\theta_t'&=e^{-at}\Delta_g(P)e^{-N\tau(x)}+(-a+2P-2)\theta\\
&=\Delta_g\theta+2Ne^{-N\tau(x)}\nabla_g(e^{-at}P)\cdot\nabla_g\tau(x)-e^{-at}P\Delta_g(e^{-N\tau(x)})+(-a+2P-2)\theta\nonumber\\
&=\Delta_g\theta+2N\nabla_g\tau(x)\cdot \nabla_g\theta+[N^2 \big|\nabla_g\tau(x)\big|_g^2+N \Delta_g\tau(x)+2P-2-a]\theta\nonumber
\end{align}
on $\overline{M}\times[t_0,t_0+\delta]$. Also, by $(\ref{equn_Dtkb})$, we obtain
\begin{align}\label{equn_thetanb}
\frac{\partial}{\partial n_g}\theta=(\eta\,-\,N\,\frac{\partial\tau(x)}{\partial n_g})\theta-e^{-at-N\tau(x)}\eta'(t)
\end{align}
on $\partial M\times[t_0,t_0+\delta]$, where $g=g(t)$ and $n_g$ is the outer normal vector field of $\partial M$ with respect to $g$. By $(\ref{equn_taunb})$ and continuity of $g(t)$, there exists $\delta>0$ small such that
\begin{align}
\frac{9}{10}\leq \frac{\partial\tau(x)}{\partial n_g} \leq \frac{11}{10}
\end{align}
on $\partial M\times[t_0,t_0+\delta]$, and hence, we now take $N>0$ large so that
\begin{align}\label{inequn_taunb2}
\eta\,-\,N\,\frac{\partial\tau(x)}{\partial n_g}<0
\end{align}
on $\partial M\times[t_0,t_0+\delta]$. Now we take $a>0$ large so that
\begin{align*}
N^2 \big|\nabla_g\tau(x)\big|_g^2+N \Delta_g\tau(x)+2P-2-a<0
\end{align*}
on $\overline{M}\times[t_0,t_0+\delta]$. By the assumption, $\theta(x,t_0)\leq0$ for $x\in \overline{M}$, and there exists $t_1\in (t_0,t_0+\delta)$ such that $\sup_{\overline{M}}\theta(\cdot,t_1)>0$. Therefore, by the maximum principle for $(\ref{equn_thetat})$, there exists $t_2\in(t_0,t_1]$ such that
\begin{align*}
\theta(r_0,t_2)=\sup_{\overline{M}\times[t_0,t_2]}\theta>0,
\end{align*}
and hence,
\begin{align*}
\frac{\partial}{\partial n_g}\theta(r_0,t_2)\geq0
\end{align*}
contradicting with $(\ref{equn_thetanb})$, by $(\ref{inequn_taunb2})$ and $\eta'(t)\geq0$. Therefore, the inequality $(\ref{inequn_P-0})$ holds on $\overline{M}\times[0,T)$.

Now if $\eta'(t)>0$ for $t\in(0,\epsilon)$, by the strong maximum principle for $(\ref{equn_n2Pt})$, $P<0$ on $M\times(0,T)$. Assume $P(r_0,t_1)=0$ for some $t_1\in(0,T)$. Thus, by the Hopf Lemma (see Lemma 2.8 on Page 12 in \cite{Liebm}), we obtain
\begin{align*}
\frac{\partial}{\partial n_g}P(r_0,t_1)>0,
\end{align*}
contradicting with $(\ref{equn_Dtkb})$. This completes the proof of the theorem.



\end{proof}

To show the long-time existence of the solution $u$ to the initial-boundary value problem $(\ref{equn_ut})-(\ref{equn_uint})$, we recall the comparison theorem derived in \cite{Li3}.
\begin{thm}(Theorem 3.1 and Remark 3.1 in \cite{Li3}) Let $u_1$ be a subsolution to $(\ref{equn_ut})$ satisfying
\begin{align*}
&(u_1)_t'\leq e^{-2u}(\Delta_{g(0)}u_1-K_{g(0)})-1,\,\,\text{in}\,\,\overline{M}\times[0,T),\\
&\frac{\partial u_1}{\partial n_{g(0)}}+k_{g(0)}\leq\eta_1 e^{u_1},\,\,\text{on}\,\,\partial M\times[0,T),\\
&u_1\big|_{t=0}\leq u_{01},\,\,\,\text{in}\,\,\overline{M},
\end{align*}
and $u_2$ be a supersolution to $(\ref{equn_ut})$ satisfying
\begin{align*}
&(u_2)_t'\geq e^{-2u_2}(\Delta_{g(0)}u_2-K_{g(0)})-1,\,\,\text{in}\,\,\overline{M}\times[0,T),\\
&\frac{\partial u_2}{\partial n_{g(0)}}+k_{g(0)}\geq\eta_2 e^{u_2},\,\,\text{on}\,\,\partial M\times[0,T),\\
&u_2\big|_{t=0}\geq u_{02},\,\,\,\text{in}\,\,\overline{M},
\end{align*}
with $u_{01}\leq u_{02}$ in $\overline{M}$ and $\eta_1\leq \eta_2$ on $\partial M\times[0,T)$. Then we have
\begin{align}
u_1\leq u_2
\end{align}
for $(x,t)\in \overline{M}\times[0,T)$.
\end{thm}
We restate the conclusion regarding the upper bound control of the solution from Theorem 1.3 in \cite{Li3} as a new theorem below:
\begin{thm}\label{thm_upperbdun2}
Let $(\overline{M},g)$ be a compact surface with its interior $M$ and boundary $\partial M$. Consider the initial-boundary value problem
\begin{align}
&\label{equn_ut2}u_t=e^{-2u}(\Delta_gu-K_g)-1,\,\,\text{in}\,\,\overline{M}\times[0,\infty),\\
&\label{equn_unb2}\frac{\partial u}{\partial n_g}+k_{g}=\eta e^u,\,\,\text{on}\,\,\partial M\times[0,\infty),\\
&u\label{equn_uint2}\big|_{t=0}=u_0,\,\,\,\text{in}\,\,\overline{M},
\end{align}
Here $g$ is a fixed background Riemannian metric continuous up to $\partial M$. We assume that $u_0\in C^{2,\alpha}(\overline{M})$ and $\eta\in C^{1+\alpha,\frac{1}{2}+\frac{\alpha}{2}}(\partial M\times [0,T_1])$ for all $T_1>0$, and also the compatibility condition holds on $\partial M\times\{0\}$:
\begin{align*}
\frac{\partial u_0}{\partial n_{g}}+k_{g}=\eta(\cdot,0)e^{u_0}.
\end{align*}
Moreover, assume that $(\ref{inequn_etaupperbd})$ holds on $\partial M\times[0,\infty)$, where $y(t)\in C^3([0,\infty))$ is some positive function satisfying $(\ref{inequn_ODEgrowth})$ for $t\in[0,\infty)$. Then for any $T_1>0$, there exists $C=C(T_1)>0$ such that the solution $u$ to $(\ref{equn_ut2})-(\ref{equn_uint2})$ satisfies
\begin{align}\label{inequn_upperbdun2}
u\leq C(T_1)
\end{align}
on $\overline{M}\times[0,T_1)$.
\end{thm}

Now we show the long-time existence of the solution $g(t)$ to the initial-boundary value problem $(\ref{equn_NRF1d2})-(\ref{equn_Ninc2d2})$.
\begin{thm}\label{thm_existlongtimen2}
Let $n=2$. Assume that $\eta=\eta(t)\in C^k([0,\infty))$ is non-decreasing for $t\geq0$ and $\eta'(t)>0$ on $(0,t_0)$ for some $t_0>0$. Moreover, $(\ref{inequn_etaupperbd})$ holds on $\partial M\times[0,\infty)$, where $y(t)\in C^3([0,\infty))$ is some positive function satisfying $(\ref{inequn_ODEgrowth})$ for $t\in[0,\infty)$.  Then there exists a unique solution $g(t)$ to $(\ref{equn_NRF1d2})-(\ref{equn_Ninc2d2})$ on $\overline{M}\times[0,\infty)$, under the compatibility conditions $(\ref{equn_comptk})$ with the mean curvature $H_{g_{-1}}\big|_{\partial M}$ replaced by the geodesic curvature $k_{g_{-1}}\big|_{\partial M}$ for some $k\geq2$.
\end{thm}
\begin{proof}
By $(\ref{inequn_P-0})$, the solution $u(r,t)=-\int_0^tP(r,\tau)d\tau$ to the initial-boundary value problem $(\ref{equn_ut})-(\ref{equn_uint})$ is increasing in $t$, for $(r,t)\in[0,r_0]\times[0,T)$. Hence, $u(r,t)\geq 0$ for $(r,t)\in[0,r_0]\times[0,T)$. That is to say, the function $v=u_t'(r,t)=-P(r,t)\geq0$, for $(r,t)\in[0,r_0]\times[0,T)$. Therefore, $u(\cdot,t)$ is a subsolution to
\begin{align}\label{equn_statlim}
\Delta_{g(0)}u-K_{g(0)}-e^{2u}=0
\end{align}
in $\overline{M}$ for $T> t\geq0$. Differentiating both sides of $(\ref{equn_ut})$ with respect to $t$ yields
\begin{align}\label{equn_vt}
v_t=e^{-2u}\Delta_{g(0)}v-2(v^2+v)
\end{align}
on $\overline{M}\times[0,T)$.

On the other hand, by Theorem \ref{thm_upperbdun2}, for any $T_1>0$, $(\ref{inequn_upperbdun2})$ holds on $\overline{M}\times[0,T_1)$.

In summary, if the largest existence time $T$ of the solution $g(t)$ satisfies $T<\infty$, then $u\geq0$ is uniformly bounded on $\overline{M}\times[0,T)$. Hence, by the standard regularity theory for parabolic equations, there exists a positive constant $C=C(T,\alpha)>0$ such that
\begin{align*}
\|u\|_{C^{4+\alpha,2+\frac{\alpha}{2}}(M\times[0,T))}\leq C(T,\alpha)
\end{align*}
for any $\alpha\in(0,1)$. Therefore, the solution $u$ can be extended to time $T$, and hence, the solution $g(t)$ exists for all $t\geq0$. This proves the theorem.

\end{proof}

Now we are ready to show the convergence of the solution $g(t)$ to $(\ref{equn_NRF1d2})-(\ref{equn_Ninc2d2})$ to a complete hyperbolic metric in $M$.

\begin{proof}[Proof of Theorem \ref{thm_convern2}.]
Since $u(\cdot,t)\in C^2(\overline{M})$ is a subsolution to $(\ref{equn_statlim})$, by the comparison theorem,
\begin{align*}
u(x,t)\leq u_{LN}(x)
\end{align*}
for $(x,t)\in M\times[0,\infty)$, where $u_{LN}$ is the solution to the Loewner-Nirenberg problem of $(\ref{equn_statlim})$ on $M$. Thus, by applying the standard regularity theory for parabolic equations to $(\ref{equn_ut})$, for each compact subset $F\subseteq M$, there exists a constant $C=C(F)>0$ such that
 \begin{align*}
 \|u\|_{C^{4,2}(F\times[0,\infty))}\leq C(F).
 \end{align*}
 Since $u(x,t)$ is increasing in $t$, $u(x,t)$ converges to $u_{\infty}(x)\leq u_{LN}(x)$ pointwisely in $M$ as $t\to\infty$. Therefore, by the Harnack inequality for $(\ref{equn_vt})$, $v=u_t\to0$ and $u(\cdot,t)\to u_{\infty}(\cdot)$ locally uniformly in $M$ as $t\to\infty$. Hence, standard regularity theory for $(\ref{equn_ut})$ yields, $u(x,t)\to u_{\infty}(x)$ locally uniformly in $C^{2,\alpha}$ in $M$. Therefore, $u_{\infty}$ satisfies $(\ref{equn_statlim})$ in $M$.

It remains to show the completeness of the metric $g_{\infty}=e^{2u_{\infty}}g(0)$ on $M$, and it suffices to show that the volume of $M$ with respect to $g(t)$ goes to infinity, as $t\to\infty$.

By $(\ref{equn_NRFa})$ and $(\ref{equn_NRFb})$,
\begin{align*}
\frac{d}{dt}Vol_g(B_{r_0}(0))=\int_{B_{r_0}(0)}-2P dV_g,
\end{align*}
for $t\geq0$. Taking derivative with respect to $t$ again, by $(\ref{equn_n2Pt})$ and $(\ref{equn_Dtkb})$ one has
\begin{align}
\frac{d}{dt}\int_{B_{r_0}(0)}-2P dV_g
&=\int_{B_{r_0}(0)}-2\Delta_gP+4P dV_g\\
&=-2\int_{\partial M}\frac{\partial P}{\partial n_g}dS_g+4\int_{B_{r_0}(0)}P dV_g\nonumber\\
&=2\int_{\partial M}\eta'(t)-\eta(t)P dS_g+4\int_{B_{r_0}(0)}P dV_g\nonumber\\
&\geq2\eta'(t)\big|\partial M\big|_{g(0)}+4\int_{B_{r_0}(0)}P dV_g,\nonumber
\end{align}
with $\big|\partial M\big|_{g(0)}$ the length of $\partial M$ with respect to $g(0)$, and hence,
\begin{align*}
\int_{B_{r_0}(0)}-2P(\cdot,t)dV_g\geq e^{2t_0-2t}\int_{B_{r_0}(0)}-2P(\cdot,t_0)dV_{g(t_0)}+2\big|\partial M\big|_{g(0)}e^{-2t}\int_{t_0}^te^{2\tau}\eta'(\tau)d\tau,
\end{align*}
for $t\geq t_0>0$. Therefore,
\begin{align*}
Vol_{g(t)}(B_{r_0}(0))&= Vol_{g(t_0)}(B_{r_0}(0))+\int_{t_0}^t\int_{B_{r_0}(0)}-2P(\cdot,\tau) dV_{g(\tau)}d\tau\\
&\geq 2\big|\partial M\big|_{g(0)}\int_{t_0}^te^{-2\tau}\int_{t_0}^{\tau}e^{2y}\eta'(y)dy d\tau\to\infty
\end{align*}
as $t\to\infty$. This completes the proof of the theorem.

\end{proof}

\begin{appendix}
\section{A proof of Theorem \ref{thm_convern2} without the condition $(\ref{inequn_etainfty})$}

By the first part of the proof of Theorem \ref{thm_convern2} in Section \ref{section_6}, it suffices to prove that the limit metric $g_{\infty}$ is complete. Recall that $K_g=-b_{ss}/b$ and
\begin{align}\label{eq_mono}
K_g\le-1,\qquad u_t\ge0,\qquad u\ge0.
\end{align}
Consequently, $b_{ss}\ge b$, so that $b_s\ge b_s(0,t)=1$. Integrating the identity
\[
 \partial_s(b_s^2-b^2)=2b_s(b_{ss}-b)\ge0
\]
from \(0\) to \(s\) yields \(b_s^2-b^2\ge1\) and therefore
\(b_s\ge\sqrt{1+b^2}\).

Set \(B(t)=b(r_0,t)\) and let $A(t)$ denote the area of $M$ with respect to $g(t)$.
The area element is \(b\,d s\,d\theta\), and strict
monotonicity of \(b\) in \(s\) allows the change of variable
\(z=b(s,t)\):
\begin{align}\label{eq_Aupper}
 A(t)=2\pi\int_0^{s(r_0,t)}b(s,t)\,d s
 \leq 2\pi\int_0^{B(t)}\frac{z}{\sqrt{1+z^2}}\,d z
  =2\pi\bigl(\sqrt{1+B(t)^2}-1\bigr).
\end{align}
Also, \(\text{Length}_{g(t)}(\partial M)=2\pi B(t)\) and
\(B(t)=e^{u(r_0,t)}\sinh r_0\ge\sinh r_0\). Therefore
\begin{align}\label{eq_ratio}
 \frac{\sqrt{1+B(t)^2}}{B(t)}
 =\sqrt{1+B(t)^{-2}}
 \le\sqrt{1+\sinh^{-2}r_0}
 =\coth r_0=\eta(0).
\end{align}

The trace of \(\partial_tg=-2(K_g+1)g\)
is \(-4(K_g+1)\), so
\(\partial_t \,d A_g=-2(K_g+1) \,d A_g\). The Gauss--Bonnet formula on the
disk, using \(k_g|_{\partial M}=\eta(t)\), gives
\[
 \int_M K_g\,d A_g+2\pi\eta(t)B(t)=2\pi.
\]
It follows that
\begin{align}
 A'(t)
 &=-2\int_M(K_g+1)\,d A_g \notag\\
 &=2\bigl(2\pi\eta(t)B(t)-2\pi-A(t)\bigr) \notag\\
 &\ge4\pi\bigl(\eta(t)B(t)-\sqrt{1+B(t)^2}\bigr) \notag\\
 &\ge4\pi B(t)\bigl(\eta(t)-\eta(0)\bigr),\label{eq_Agrowth}
\end{align}
where the inequalities use \eqref{eq_Aupper} and \eqref{eq_ratio},
respectively. Now we take $t_1\in(0,t_0)$. Therefore, for $t\geq t_1$, we have
\begin{align}\label{eq_Alinear}
 A(t)\ge A(t_1)+
   4\pi\,(t-t_1)\,\bigl(\eta(t_1)-\eta(0)\bigr)\sinh(r_0).
\end{align}
In particular, \(A(t)\to\infty\). This completes the proof of Theorem 1.3 without the condition $(1.18)$.

\end{appendix}

 \end{document}